\documentclass{amsart}

\usepackage{amsmath}
\usepackage{tikz}
\usepackage{booktabs}

\usetikzlibrary{shapes.geometric}

\usepackage{amssymb}
\usepackage{amscd}

\title [Surfaces with no convex presentation]{Translation surfaces with no convex presentation}

\author{Samuel Leli\`evre}
\address{Laboratoire de math\'ematique d'Orsay, Umr 8628 Cnrs/Universit\'e 
Paris-Sud, B\^at.~425, campus Orsay, 91405 Orsay cedex, France,
{\tt samuel.lelievre@math.u-psud.fr}}

\author{Barak Weiss}

\address{Ben Gurion University, Be'er Sheva, Israel 84105
{\tt barakw@math.bgu.ac.il}}

\newif\ifdraft\drafttrue

\draftfalse

\font\sb = cmbx8 scaled \magstep0

\font\sn = cmssi8 scaled \magstep0

\long\def\combarak#1{\ifdraft{\sb #1 }\else\ignorespaces\fi}

\newcommand\name[1]{\label{#1}{\ifdraft{\sn [#1]}\else\ignorespaces\fi}}

\newcommand\eq[2]{{\ifdraft{\ \tt [#1]}\else\ignorespaces\fi}\begin{equation}\label{eq:
#1}{#2}\end{equation}}

\newcommand {\equ}[1]     {\eqref{eq: #1}}

\newcommand{\Q}{{\mathbb {Q}}}

\newcommand{\R}{{\mathbb{R}}}

\newcommand{\Z}{{\mathbb{Z}}}

\newcommand{\HH}{{\mathcal{H}}}

\newcommand{\EE}{{\mathcal{E}}}

\newcommand{\SL}{\operatorname{SL}}

\newcommand{\conv}{{\rm conv}}

\newcommand{\df}{{\, \stackrel{\mathrm{def}}{=}\, }}

\newcommand{\x}{{\bf x}}

\newcommand{\supp}{{\rm supp}}

\newcommand{\sm}{\smallsetminus}

\newcommand{\vre}{\varepsilon}

\font\sb = cmbx8 scaled \magstep0

\newcommand\hol{\mathrm{hol}}

\newcommand{\Hnc}{{
\HH_{\mathrm{nc}}}} 

\newcommand{\Hnsc}{{
\HH_{\mathrm{nsc}}}}

\newcommand {\ignore}[1]  {}

\newtheorem{thm}{Theorem}

\newtheorem{prop}[thm]{Proposition}

\newtheorem{cor}[thm]{Corollary}

\newtheorem{remark}[thm]{Remark}

\begin{document}

\begin{abstract}
We give infinite lists of translations surfaces with no convex
presentations. 
We classify the surfaces in the stratum $\mathcal{H}(2)$ which do not have convex
presentations, as well as those with no strictly convex
presentations. We show that in $\HH(1,1)$, all surfaces in the
eigenform loci $\EE_4$, $\EE_9$ or $\EE_{16}$ have no
strictly convex presentation, and that the list of surfaces with no
convex presentations in
$\HH(1,1) \sm (\EE_4 \cup \EE_9 \cup \EE_{16})$ is finite and
consists of square-tiled surfaces. We prove the existence of non-lattice surfaces
without strictly convex presentations in all of the strata $\HH^{\mathrm{(hyp)}}(g-1,
g-1)$. 

\end{abstract}

\maketitle
A {\em translation surface} is a union of polygons with pairs of parallel
edges identified by translation, up to cut and paste equivalence. These structures have
been intensively studied in recent years in connection with questions
in group theory, geometry, complex analysis, and dynamics --- 
we recall the definitions in \S \ref{sec: general} and  refer to the
recent surveys \cite{MT, zorich survey} for detailed 
definitions and discussions of work on translation surfaces. The
cut and paste equivalence means that the same translation surface 
has many presentations as unions of polygons.
It is well-known (see \cite[\S12]{Viana survey}) that
a connected translation surface may be presented as a single polygon
with an even number of edges identified pairwise. A natural question
is whether there is such a presentation, in which the polygon can be
taken to be convex. In his 1992 Hayashibara Forum Lecture,  Veech \cite{Veech hyperelliptic} was the
first to exhibit surfaces with no presentations as convex polygons. We
review his examples and give new ones in \S \ref{sec: examples}. 
These include the Veech double $n$-gons, the
Escher staircases, the Ward examples, and the Bouw-M\"oller examples.

Translation surfaces are naturally grouped in strata, which are moduli
spaces of translation surfaces for which the combinatorics of
singularities is fixed. Although the constructions in \S \ref{sec: examples} 
yield infinitely many examples, they  only give rise to finitely
many in each stratum. In this paper we 
take up the question of classifying 
surfaces with no convex presentations in the simplest strata, namely
the genus two strata
$\HH(2)$ and $\HH(1,1)$. It
will  be useful to distinguish between {\em convex} and {\em strictly
  convex} polygons. Although, strictly speaking, polygons are never
strictly convex, we will say that a polygon is strictly
convex if none of its vertices is in the convex hull of the other
vertices; equivalently, the internal angles at vertices of the polygon are
strictly less than $\pi$. In making this definition we always assume that polygons
have no spurious vertices, in other words, edges of polygons are
not strictly contained in larger straight segments on which the
identification map of edges extends continuously. We will see that
strictly convex presentations only arise in the hyperelliptic
components of the strata $\HH(2g-2)$ and $\HH(g-1, g-1)$.

Let $G \df \SL_2(\R)$. There is a natural $G$-action on each stratum
of translation surfaces, and, as noted by Veech \cite{Veech
  hyperelliptic}, the property of having no (strictly) 
convex presentations is $G$-invariant. Moreover the property of having
no strictly convex presentations is closed. Thus the question of
classifying surfaces without convex presentations is intimately
connected to the question of understanding the closures of
$G$-orbits. Our results crucially rely on the work of McMullen \cite{McMullen-SL(2)} 
describing the possible orbit closures for the $G$-action on $\HH(2)$
and $\HH(1,1)$. 

We
remind the reader that by work of Calta \cite{Calta} and McMullen
\cite{McMullen JAMS, McMullen disc} (see \S \ref{sec:
genus 2} for a review), the 
$G$-orbits of 
lattice surfaces in $\HH(2)$ are in bijective correspondence with 
symbols $D, D_0, D_1$, where $D \geq 5$ is the discriminant, and is
congruent to a square
mod 4, and where for $9< D \equiv 1 \mod 8$ there are symbols $D_0,
D_1$ corresponding to even and odd spin structures. 
Some surfaces in $\HH(1,1)$ have Jacobians which are eigenforms for an
action of an order of discriminant $D$ in the field $\Q(\sqrt{D})$. 
We will denote the locus of such surfaces in $\HH(1,1)$
by $\EE_D$. In this paper, by an  {\em arithmetic surface} we mean a
translation cover of a torus branched over one point, i.e. a surface
which is in the $G$-orbit of a square-tiled surface. Note that lattice
surfaces which are eigenforms of discriminant $D$, where $D$ is a
square, are arithmetic. The following are our main results. 

\begin{thm}
\name{thm: main}
Up to the $G$-action, the list of surfaces in
$\mathcal{H}(2)$ which have no strictly convex presentation are those
corresponding to symbols 
\eq{eq: list}{
5, 9, 12, 16, 17_1, 21, 25_0, 25_1, 32, 36, 41_0, 45,
49_1, 64, 77, 81_1.
}
Up to the $G$-action, the list of surfaces in
$\HH(2)$ with no convex presentation are those corresponding to symbols 
\eq{eq: list2}{
5, 12, 17_1, 21, 32, 41_0, 45, 77
.}
\end{thm}

\begin{thm}\name{thm: main H(1,1)}
Let $\EE_{4,9,16} \df \EE_4 \cup \EE_9 \cup \EE_{16} \subset \HH(1,1)$. Then:
\begin{enumerate}
\item[(i)]
There is no surface in $\EE_{4,9,16}$ with a strictly convex
presentation.
\item[(ii)]
If $D \notin \{4,9,16\} \, (D \equiv 0,1 \mod 4),$ then the eigenform locus $\EE_D$ contains a
surface with a 
strictly convex presentation.  
\item[(iii)]
 Any surface in $\HH(1,1)$ which has no strictly convex
presentation is either arithmetic or belongs to $\EE_{4, 9, 16}$. 
\item[(iv)]
The number of $G$-orbits of arithmetic surfaces with no strictly convex
presentations in $\HH(1,1) \sm \EE_{4, 9, 16}$ is
finite. 

\item[(v)]\name{item: answer to Veech}
For each $D \in \{4,9,16\}$, there are surfaces in $\EE_D$ which are not lattice
surfaces and have no 
convex presentations.
\end{enumerate}

\end{thm}

The space $\EE_{4,9,16}$ parametrizes torus covers of low degree, and
assertion (i) is a special case of a more general statement (see
Proposition \ref{remark: more general}) which exhibits loci of torus covers with no strictly
convex presentation for any genus $g \geq 2$. 
Assertion (v)
resolves a question of Veech
\cite[Question 5.2]{Veech hyperelliptic}. Using Theorem \ref{thm: main
H(1,1)}, in \S
\ref{sec: another question} we answer a more refined
question of Veech. Our proof of
(iii) is ineffective, and it would be interesting to write down the
full list 
of examples of the arithmetic surfaces as in (iii). We exhibit two 
such examples, in discriminant 49, in \S \ref{subsec: square zoo}. 

\medskip

We briefly describe the proofs of our results. Clearly a
convex polygon contains a strip running through it which connects
parallel edges on opposite sides of the polygon. When $g \geq 2$, this
shows that a translation surface with a convex presentation has a
simple cylinder, i.e. a cylinder which is bounded on 
both sides by a single saddle connection. The existence of a simple
cylinder is a nontrivial restriction and this observation underlies
our construction (in \S
\ref{sec: examples}) of 
many new examples of surfaces with no convex presentations.  In genus two, the dynamical results of
McMullen allow us to conclude that a surface with no convex
presentation is either a lattice surface or an eigenform. By further
results of Calta and McMullen, the core direction of the simple
cylinder is a completely periodic direction. By work of McMullen,
Calta and Bainbridge, completely periodic directions on genus two
surfaces are described in terms of prototypes, which are explicit
presentations of completely periodic surfaces. There are finitely many
prototypes for each discriminant $D$ in $\HH(2)$, and in $\HH(1,1)$,
prototypes also depend on two continuous real parameters. In \S
\ref{sec: nec and suff} we show that there is a canonical polygonal 
presentation associated with each prototype, which may or may not be
convex. We call this the {\em canonical polygon} associated with the
prototype. A strictly convex presentation, if it exists, must be the
canonical polygon of some prototype. 

To prove Theorem \ref{thm: main}, we first exhibit explicitly, for all $D \geq 200$,
a particular prototype whose canonical polygon is convex. 
For the remaining cases $5 \leq D \leq 200$ we use a computer program,
discussed in \S \ref{subsec: computer algorithm}, to list all prototypes
and test their canonical polygon for convexity. Figure \ref{fig:
  another figure} shows two prototypes in $D=12$ and $D=13$ and shows
the corresponding canonical polygon. Using this program we
obtain the list \equ{eq: list}. Sample output of our computer program
is shown in Figure \ref{fig: sample output}. 
 Appealing to results of Hubert,
Leli\`evre and McMullen we find that lattice surfaces which are
arithmetic do admit a one-cylinder convex presentation. We use this to
obtain the list \equ{eq: list2}. 

\begin{figure}[!h]
\begin{tikzpicture}
\begin{scope}[xshift=2cm,yshift=0cm] 
\draw [->] (2.75,-0.65) -- +(1.25,0);
\coordinate (A) at (0,0);
\coordinate (B) at (0,-1.73205);
\coordinate (C) at (1.73205,-1.73205);
\coordinate (D) at (1.73205,0);
\coordinate (E) at (3,0);
\coordinate (F) at (4,1);
\coordinate (G) at (2.73205,1);
\coordinate (H) at (1,1);
\coordinate (GE) at (2.366,0);
\coordinate (EG) at (3.366,1);
\coordinate (G1) at (2.73205,0.);
\coordinate (G2) at (3.73205,1);
\coordinate (G3) at (3.73205,0);
\coordinate (G4) at (0.73205,1);
\coordinate (G5) at (0.73205,0);
\fill [gray!10] (A) -- (D) -- (H) -- cycle;
\fill [gray!20] (G) -- (GE) -- (E) -- cycle;
\fill [gray!15] (E) -- (EG) -- (F) -- cycle;
\draw [gray] (H) -- (D);
\draw [gray!50] (D) -- (G);
\draw [gray] (E) -- (EG);
\draw [gray] (GE) -- (G);
\draw [gray] (G) -- (E);
\draw (A) -- (D);
\draw [semithick, densely dotted] (H) -- (1,0);
\draw [semithick, densely dotted] (G) -- (G1);
\draw [semithick, densely dotted] (G2) -- (intersection of E--F and G2--G3);
\draw [semithick, densely dotted] (intersection of A--H and G4--G5) -- (G5);
\draw (A) node {\scriptsize\textbullet}
--
(B) node {\scriptsize\textbullet}
--
(C) node {\scriptsize\textbullet}
--
(D) node {\scriptsize\textbullet}
--
(E) node {\scriptsize\textbullet}
--
(F) node {\scriptsize\textbullet}
--
(G) node {\scriptsize\textbullet}
--
(H) node {\scriptsize\textbullet}
--
(0,0) -- cycle;
\end{scope}
\begin{scope}[xshift=7cm,yshift=0cm] 
\coordinate (A) at (0,0);
\coordinate (B) at (0,-1.73205);
\coordinate (C) at (1.73205,-1.73205);
\coordinate (D) at (1.73205,0);
\coordinate (G) at (0.73205,2);
\coordinate (H) at (1,1);
\coordinate (HH) at (1,1);
\coordinate (HH) at (0.73205,-2.73205);
\coordinate (GG) at (1,-3.73205);
\coordinate (H1) at (0,1);
\coordinate (HH1) at (1.73205,-2.73205);
\coordinate (Hl) at (intersection of A--G and H1--H);
\fill [gray!10] (A) -- (D) -- (H) -- cycle;
\fill [gray!15] (A) -- (H) -- (Hl) -- cycle;
\fill [gray!20] (Hl) -- (H) -- (G) -- cycle;
\draw [gray] (A) -- (H);
\draw (B) -- (C);
\draw [gray] (intersection of A--G and H1--H) -- (H);
\draw [gray] (intersection of GG--C and HH1--HH) -- (HH);
\draw [gray!50] (C) -- (HH);
\draw (A) -- (D);
\draw [semithick, densely dotted] (H) -- (1,0);
\draw [semithick, densely dotted](G) -- (0.73205,0.);
\draw (A) node {\scriptsize\textbullet}
--
(B) node {\scriptsize\textbullet}
--
(HH) node {\scriptsize\textbullet}
-- (GG) node {\scriptsize\textbullet}
-- (C) node {\scriptsize\textbullet}
--
(D) node {\scriptsize\textbullet}
--
(H) node {\scriptsize\textbullet}
--
(G) node {\scriptsize\textbullet}
-- (0,0) -- cycle;
\end{scope}
\begin{scope}[xshift=0cm,yshift=-3.5cm] 
\draw [->] (2.3,-0.65) -- +(1.7,0);
\coordinate (A) at (0,0);
\coordinate (B) at (0,-1.3028);
\coordinate (C) at (1.3028,-1.3028);
\coordinate (D) at (1.3028,0);
\coordinate (E) at (3,0);
\coordinate (F) at (4,1);
\coordinate (G) at (2.3028,1);
\coordinate (H) at (1,1);
\coordinate (GE) at (2.1514,0);
\coordinate (EG) at (3.1514,1);
\coordinate (G1) at (2.3028,0.);
\coordinate (G2) at (3.3028,1);
\coordinate (G3) at (3.3028,0);
\coordinate (G4) at (0.3028,1);
\coordinate (G5) at (0.3028,0);
\fill [gray!10] (A) -- (D) -- (H) -- cycle;
\fill [gray!20] (G) -- (GE) -- (E) -- cycle;
\fill [gray!15] (E) -- (EG) -- (F) -- cycle;
\draw [gray] (H) -- (D);
\draw [gray!50] (D) -- (G);
\draw [gray] (E) -- (EG);
\draw [gray] (GE) -- (G);
\draw [gray] (G) -- (E);
\draw (A) -- (D);
\draw [semithick, densely dotted] (H) -- (1,0);
\draw [semithick, densely dotted] (G) -- (G1);
\draw [semithick, densely dotted] (G2) -- (intersection of E--F and G2--G3);
\draw [semithick, densely dotted] (intersection of A--H and G4--G5) -- (G5);
\draw (A) node {\scriptsize\textbullet}
--
(B) node {\scriptsize\textbullet}
--
(C) node {\scriptsize\textbullet}
--
(D) node {\scriptsize\textbullet}
--
(E) node {\scriptsize\textbullet}
--
(F) node {\scriptsize\textbullet}
--
(G) node {\scriptsize\textbullet}
--
(H) node {\scriptsize\textbullet}
--
(0,0) -- cycle;
\end{scope}
\begin{scope}[xshift=5cm,yshift=-3.5cm] 
\coordinate (A) at (0,0);
\coordinate (B) at (0,-1.3028);
\coordinate (C) at (1.3028,-1.3028);
\coordinate (D) at (1.3028,0);
\coordinate (G) at (0.3028,2);
\coordinate (H) at (1,1);
\coordinate (HH) at (1,1);
\coordinate (HH) at (0.3028,-2.3028);
\coordinate (GG) at (1,-3.3028);
\coordinate (H1) at (0,1);
\coordinate (HH1) at (1.3028,-2.3028);
\coordinate (Hl) at (intersection of A--G and H1--H);
\fill [gray!10] (A) -- (D) -- (H) -- cycle;
\fill [gray!15] (A) -- (H) -- (Hl) -- cycle;
\fill [gray!20] (Hl) -- (H) -- (G) -- cycle;
\draw [gray] (A) -- (H);
\draw (B) -- (C);
\draw [gray] (intersection of A--G and H1--H) -- (H);
\draw [gray] (intersection of GG--C and HH1--HH) -- (HH);
\draw [gray!50] (C) -- (HH);
\draw (A) -- (D);
\draw [semithick, densely dotted] (H) -- (1,0);
\draw [semithick, densely dotted](G) -- (0.3028,0.);
\draw (A) node {\scriptsize\textbullet}
--
(B) node {\scriptsize\textbullet}
--
(HH) node {\scriptsize\textbullet}
-- (GG) node {\scriptsize\textbullet}
-- (C) node {\scriptsize\textbullet}
--
(D) node {\scriptsize\textbullet}
--
(H) node {\scriptsize\textbullet}
--
(G) node {\scriptsize\textbullet}
-- (0,0) -- cycle;
\end{scope}
\end{tikzpicture}
\caption{Canonical polygon of prototype $(D,a,b,c,e)$:\newline
Prototype  $(12,1,3,1,0)$ yields a non-convex
polygon. 
\newline
Prototype $(13,1,2,1,-1)$ yields a strictly convex
polygon.}\name{fig: another figure}
\end{figure}
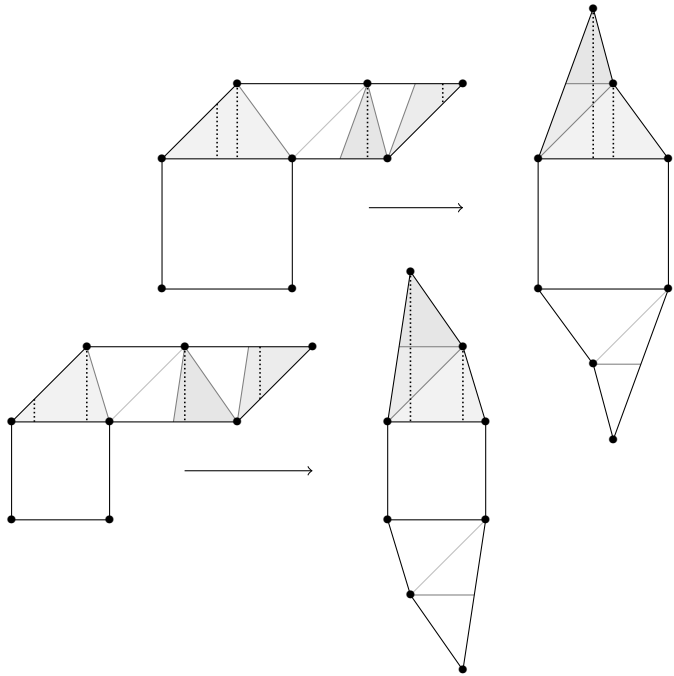

To prove Theorem \ref{thm: main H(1,1)}(i) we use the fact that
surfaces in $\EE_{4,9,16}$ admit a branched covering map of a twice punctured
torus, of low degree, to rule out strictly
convex canonical polygons. Note that in $\HH(1,1)$ the number of
prototypes for a fixed $D$ is no longer finite so we cannot use a brute force computer
search as in the case $\HH(2), D \leq 200.$ For Theorem \ref{thm: main
  H(1,1)}(ii), we discuss separately the discriminants $D$ which do or
do not 
appear in the list \equ{eq: list}.  If $D$ does not
appear in \equ{eq: list}, the existence of a strictly
convex polygon in $\HH(2)$ forces the existence of a nearby strictly
convex polygon of discriminant $D$ in $\HH(1,1)$. For the remaining cases
of $D \neq 4, 9,16$ which do appear in \equ{eq: list} we give explicit
strictly convex canonical polygons for some prototypes (these were
found by a computer search). Assertion (iii) of Theorem \ref{thm:
  main H(1,1)} follows from assertion (ii) and 
McMullen's dynamical results
for the $G$-action on $\HH(1,1)$. 

The proof of the finiteness statement in Theorem \ref{thm: main
  H(1,1)}(iv) relies on a deep recent finiteness result of 
\cite{EM, EMM}. To prove Theorem \ref{thm: main
  H(1,1)}(v) we show that convex but not strictly convex presentations
may be `easily destroyed' by a local perturbation in $\EE_D$. Since
there are no strictly convex presentations in $\EE_{4,9,16} $,
this implies that the set of surfaces in $\EE_{4, 9, 16}$ with no
convex presentations at all must be of full measure, with respect to
the natural measure on these loci. 

\medskip

{\bf Acknowledgments:}  
Our interest in the question of convex presentations of translation
surfaces was inspired by Thierry Monteil's `banana peel argument'
\cite{Monteil}. 
We thank Matt Bainbridge for useful discussions. We thank Pat Hooper
for permission to use his picture  (Figure \ref{fig: BM}) of the Bouw-M\"oller
surfaces, and for useful discussions about these
surfaces. 
This research was supported by the ANR projet blanc GEODYM, Israel Science Foundation
grant 190/08 and
European Research Council grant DLGAPS 279893.
\section{General remarks}\name{sec: general}

A {\em cylinder} for a translation surface is a topological annulus
which is isometric to $(0, h) \times \R/c\Z$, and is not properly
contained in a larger cylinder. Here $h$ and $c$ are
respectively the {\em height} and {\em circumference} of the
cylinder. The image of $\{x_0\} \times \R/c\Z$ under such an isometry
is called a {\em waist curve.} A cylinder is always bounded by saddle connections parallel
to the waist curve. A cylinder is called {\em simple} if both of its boundary
consists of a single saddle connection. 

\begin{prop}\name{prop: central symmetry}
Suppose $M$ is a translation surface of genus at least two, which has a convex
presentation as a $2n$-gon. Then $M$ contains a simple cylinder
$C$, and there is a direction $\theta$ transverse to the circumference
of $C$, such that $\theta$ is the direction of a saddle connection
contained in $C$ passing between its boundary edges (i.e. a
`diagonal'), and any path in $M$ in
direction $\theta$ either lies in a saddle connection or intersects both boundary components of $C$. 

Furthermore, if the $2n$-gon is strictly convex, then there is a
hyperelliptic involution of $M$ fixing $C$, all paths in $M$ in
direction $\theta$ intersect any boundary component of $E$ of $C$, and 
the return map $E \to E$ obtained by moving in direction
$\theta$ is an interval exchange on $n-1$ intervals, corresponding to the `inversive
permutation' $i \mapsto n-i$. 
\end{prop}

\begin{proof}
Denote the convex $2n$-gon representing $M$ by $P$. Since $M$ has genus
at least two, $n \geq 4$. Let  $e_1, \ldots, e_{2n}$ denote the
edges of $P$, in cyclic counterclockwise order. Let $\theta_i$ denote
the angle that the edge $e_i$ makes with the horizontal direction where the $e_i$ are oriented so that $P$ is on their left. We consider
$\theta_i$ modulo $2\pi$ and write $\theta' < \theta$ if the short
path along the boundary of the circle of directions from $\theta'$ to
$\theta$ is in the counterclockwise direction. Then convexity (strict
convexity) is
equivalent to $\theta_i \leq \theta_{i+1}$ (resp. $\theta_i <
\theta_{i+1}$) for all $i$. For each $i$ we denote by $k=k(i)$ the index
for which $e_i$ and $e_k$ are identified in $M$. 

The convex hull of $e_j$ and $e_k, \, k=k(j)$ is a parallelogram in the
plane. Denote its sides $e_j, u_j, e_k, v_j$ in the usual order. This
parallelogram maps to a cylinder $C$ in $M$, which is simple if and only
if the sides $u_j$ and $v_j$ contain no vertices of $P$ in their
interior. It is clear from convexity that any path in $M$ in direction
$\theta = \theta_j$ intersects the boundary edges of $C$, so to prove
the first assertion we have only to find $j$ for which $C$ is simple. 

  Suppose $C$ is not simple and suppose with no loss
of generality that $u_j$ contains
vertices in its interior, then $k-j \geq 3$ and 
\eq{eq: we see}{\theta_j < \theta_{j+1} = \cdots
= \theta_{k-1} < \theta_k.}

Assume first that  $k-j \geq 4$, that is an edge of $P$ has three
consecutive parallel edges. Then replacing $j$ with $j+2$, we see from
\equ{eq: we see} that neither $u_{j+2}$ nor $v_{j+2}$
can contain vertices in their interior, i.e. the parallelogram $\conv
( e_{j+2}, e_{k(j+2)})$ maps to a simple cylinder in $M$, as required. Now assume
$k=j+3$. Then the equations
$$0 \leq \theta_{i+1} - \theta_{i} < \pi, \ \theta_{j+1} =
\theta_{k(j+1)}, \ \theta_j = \pi+\theta_{j+3}, \ \theta_{i} = \pi+
\theta_{k(i)}$$
imply that the angles $\theta_i$ assume only the four values $\{\theta_j,
\theta_{j+1}, \pi+ \theta_j, \pi+\theta_{j+1}\}$, i.e. $P$ is a
parallelogram. If it has three consecutive parallel edges we are done
as in the previous case. Otherwise $n=4$ and in each direction there are two
consecutive parallel edges. Since there are no spurious vertices on
$P$, in this case the gluing must be $k(i)= i+ 4 \mod 8$, and one
sees that for any $j$, $\conv(e_j, e_{k(j)})$ maps to a simple
cylinder, as required.

Now suppose $P$ is strictly convex. We claim that 
\eq{eq: the claim}{
\forall j,\  \ k(j) = j+n \mod
2n.}
Indeed, if this is not true then (replacing $j$ with $k(j)$ and re-indexing cyclically) we would have for
some $r< n$:
\eq{eq: for a contradiction}{
\theta_1 < \cdots < \theta_r = \pi+\theta_1  < \cdots < \theta_{2n} < \theta_1.
}
On the other hand, since $r< n$ there is some $j_0$ so that both $j_0$ and $k(j_0)$ are in $\{r+1,
\ldots, 2n\}$. This implies $\theta_{j_0} = \pi+ \theta_{k(j_0)},$
contradicting \equ{eq: for a contradiction}. The remaining
assertions of the Proposition follow easily from \equ{eq: the claim},
where the involution is the symmetry with respect to the center of the
parallelogram representing the cylinder $C$. 
\end{proof}

The following is an immediate consequence:
\begin{cor}\name{cor: hyperelliptic stratum}
Suppose that $M$ is a translation surface with a presentation as a
strictly convex $2n$-gon, with $n \geq 4$. Then:
 \begin{itemize}
\item
All vertices of the $2n$-gon are singularities for the translation
structure. 
\item
If $n$ is even then $M$ belongs to the hyperelliptic connected component of the stratum
$\HH(2g-2)$ where $\displaystyle{g = \frac{n}{2}}$ is the genus of $M$. 
\item
If $n$ is odd then $M$ belongs to the hyperelliptic connected component of the stratum
$\HH(g-1, g-1)$ where $\displaystyle{g = \frac{n-1}{2}}$ is the genus
of $M$. 

\end{itemize}
\end{cor}
\qed

\begin{remark}
It was proved by Kontsevich and Zorich \cite{KZ} that in every connected
component of every stratum there
is a dense set of surfaces with a one-cylinder presentation. In
particular having a convex presentation places no constraint on the
topology of the surface. Corollary \ref{cor: hyperelliptic stratum} shows that having a strictly
convex presentation does entail a topological restriction.   
\end{remark}

We will need another consequence of Proposition \ref{prop: central
  symmetry}. 
\begin{cor}\name{cor: short saddle connection}
Suppose $M$ is a translation surface of genus $g \geq 2$ represented
by a strictly convex polygon 
$P$ with opposite sides identified, and suppose $C$ is
a simple cylinder in $M$ joining opposite sides 
of $P$. Let $\delta$ be the saddle connection joining the top and
bottom of $C$ which is represented by two opposite segments on the
boundary of $P$. Then any saddle
connection parallel to $\delta$ is longer than $\delta$, and $M$ has
no other simple cylinder which is parallel to $C$  and contains a diagonal parallel to $\delta$. 

In particular, a surface with a strictly convex presentation does not
have a nontrivial translation automorphism. 
\end{cor}
\begin{proof}
Assume with no loss of
generality that $C$ is represented by a central square in $P$ with 
opposing vertical edges identified. Since $g \geq 2, \, P$ does not
cover the entire surface $M$, and since $P$ is strictly convex, any vertical saddle
connection in $P$ must pass through $C$ at least once and have
endpoints disjoint from the top and bottom edges of $C$; in particular
its length is greater than the height of $C$. Thus, if $C'$ is another
cylinder in $M$ which is parallel to $C$, with a diagonal $\delta'$ parallel to
$\delta$, the interior of $\delta'$ must intersect $C$, that is the
interiors of $C$ and $C'$ intersect. This is not possible for parallel
cylinders. 

This proves the first assertion, and the second assertion follows
by considering the image of $C$ under a translation automorphism. 
\end{proof}

The following simple observation, noted by Veech
\cite{Veech hyperelliptic}, is at the heart of our discussion. Its
proof is immediate from the definitions and is left as an exercise to
the reader.  
\begin{prop}\name{prop: G invariant closed}
Let $\HH$ be a connected component of a stratum of translation
surface. Then the subsets $\Hnc, \Hnsc$ of $\HH$ 
consisting respectively of surfaces with no convex (resp., no
strictly convex) presentations are $G$-invariant, and $\Hnsc$ is closed.

\end{prop}
\qed

\section{A zoo of examples}\name{sec: examples}
In this section we give four infinite lists of examples of surfaces which do not admit convex
presentations. These will all be {\em lattice surfaces}, i.e. surfaces
$M$ for which the $G$-orbit $GM$  is closed. Equivalently, their {\em
  Veech group} $\Gamma = \{g \in G: gM=M\}$  is a
lattice in $G$. As Veech showed in \cite{Veech - alternative}, on a
lattice surface, the direction of a core curve of a cylinder is always
{\em completely periodic}, all non-critical leaves in that direction
are closed. Moreover the $\Gamma$-orbits of completely periodic
directions are in bijective correspondence with the cusps of the
quotient $G/\Gamma$, i.e. with conjugacy classes of maximal parabolic
subgroups of $\Gamma$. In proving that a lattice surface has no convex
presentation, our strategy will be to examine the various completely
periodic directions on $M$, and show that they cannot be the direction
of a core curve of a simple cylinder as in Proposition \ref{prop:
  central symmetry}.

We start with the {\em Veech double
  $n$-gon}, obtained by taking $n \geq 5$ and forming two regular $n$-gons $\Delta_1,
\Delta_2$, where $\Delta_2$ is the reflection of $\Delta_1$ in (any) one of
its side, and gluing each edge of $\Delta_1$ with the parallel edge of
$\Delta_2$. See Figure \ref{fig: double n gon}. 

\ignore{

\begin{figure}
\begin{tikzpicture}
\draw (0,0)
  -- ++(-18:2cm) -- ++(54:2cm) -- ++(126:2cm) -- ++(198:2cm)
  -- ++(-18:-2cm) -- ++(54:-2cm) -- ++(126:-2cm) -- ++(198:-2cm) -- cycle;
\draw (0,0) -- +(90:2cm);
\end{tikzpicture}
\caption{Double $(2g+1)$-gon. The unique completely periodic direction
is vertical.} \label{fig: double odd n gon}
\end{figure}

\begin{figure}
\begin{tikzpicture}
\draw (0,0)
  -- ++(-30:2cm) -- ++(30:2cm) -- ++(90:2cm) -- ++(150:2cm) -- ++(210:2cm)
  -- ++(-30:-2cm) -- ++(30:-2cm) -- ++(90:-2cm) -- ++(150:-2cm) -- ++(210:-2cm)
  -- cycle;
\draw (0,0) -- +(90:2cm);
\end{tikzpicture}
\caption{Double $(2g+2)$-gon. The completely periodic directions are
  vertical and horizontal.} \label{fig: double even n gon}
\end{figure}
}

\begin{figure}[!h]
\begin{tikzpicture}
\foreach \l in {1.7cm} {
\draw (0,-\l/2)
  -- ++(-18:\l) -- ++(54:\l) -- ++(126:\l) -- ++(198:\l)
  -- ++(-18:-\l) -- ++(54:-\l) -- ++(126:-\l) -- ++(198:-\l) -- cycle;
\draw (0,-\l/2) -- +(90:\l);
\path (0,-\l/2) node [inner sep=2pt,circle,draw,fill=white] {}
  ++(-18:\l) node [inner sep=2pt,circle,draw,fill=white] {}
  ++(54:\l) node [inner sep=2pt,circle,draw,fill=white] {}
  ++(126:\l) node [inner sep=2pt,circle,draw,fill=white] {}
  ++(198:\l) node [inner sep=2pt,circle,draw,fill=white] {}
  ++(-18:-\l) node [inner sep=2pt,circle,draw,fill=white] {}
  ++(54:-\l) node [inner sep=2pt,circle,draw,fill=white] {}
  ++(126:-\l) node [inner sep=2pt,circle,draw,fill=white] {}
  ;
}
\begin{scope}[xshift=6.25cm]
\foreach \l in {1.5cm} {
\draw (0,-\l/2)
  -- ++(-30:\l) -- ++(30:\l) -- ++(90:\l) -- ++(150:\l) -- ++(210:\l)
  -- ++(-30:-\l) -- ++(30:-\l) -- ++(90:-\l) -- ++(150:-\l) -- ++(210:-\l)
  -- cycle;
\draw (0,-\l/2) -- +(90:\l);
\path (0,-\l/2) node [inner sep=2pt,circle,draw,fill=white] {}
  ++(-30:\l) node [inner sep=2pt,circle,fill=black] {}
  ++(30:\l) node [inner sep=2pt,circle,draw,fill=white] {}
  ++(90:\l) node [inner sep=2pt,circle,fill=black] {}
  ++(150:\l) node [inner sep=2pt,circle,draw,fill=white] {}
  ++(210:\l) node [inner sep=2pt,circle,fill=black] {}
  ++(-30:-\l) node [inner sep=2pt,circle,draw,fill=white] {}
  ++(30:-\l) node [inner sep=2pt,circle,fill=black] {}
  ++(90:-\l) node [inner sep=2pt,circle,draw,fill=white] {}
  ++(150:-\l) node [inner sep=2pt,circle,fill=black] {}
  ;
}
\end{scope}
\end{tikzpicture}
\caption{In a double $(2g+1)$-gon, the horizontal and vertical directions
are in the unique class of parabolic directions, while in a double
$(2g+2)$-gon, they represent the two classes of parabolic directions.}
\label{fig: double n gon}
\end{figure}
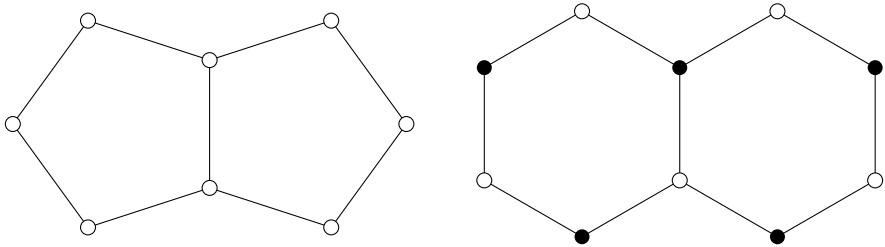

\begin{prop}\name{prop: regular polygons}
The Veech double $n$-gons have no convex presentations.   
\end{prop}

\begin{remark}
Veech \cite{Veech hyperelliptic} proved this assertion for $n$
odd and asserted its validity for all $n$, without proof. Our argument
is different from that of Veech.  
\end{remark} 

\begin{proof}
Veech \cite{Veech - alternative} showed that the double $n$-gon
surfaces are lattice surfaces, and computed 
their Veech group, showing that it has either one cusp (when $n$ is
odd) or two (when $n$ is even). 
If $n$ is odd,  in the cylinder decomposition
corresponding to the unique cusp, the only simple cylinder lies at the
left and right sides of the double
$n$-gon (vertical direction in Figure \ref{fig: double n gon}) and this cylinder does not intersect all
leaves in the direction of the
corresponding diagonal. Thus the necessary condition of Proposition
\ref{prop: central symmetry} is not satisfied. 
If $n$ is even then one cylinder decomposition has no simple
cylinders at all, and the other has one simple cylinder, treated as in the previous case of
$n$ odd. 
\end{proof}

 Now let $n \geq 3$. The
{\em Escher staircase} is obtained by a cyclic gluing 
of squares as in Figure \ref{fig: Escher}. That is, if the squares are
labelled $1, \ldots, 2n$, then square $2k+1$ is attached to square $2k+2$ along
its horizontal sides, and to square $2k$ along its two vertical 
sides (all calculations modulo $2n$).


\begin{figure}[!h]
\begin{tikzpicture}[scale=0.6]
\useasboundingbox (0,0) rectangle (6,6);
\draw (0,0)
  -- node [above] {$a$} ++(2,0)
  -- ++(2,0) -- ++(0,2)
  -- ++(2,0) -- ++(0,2)
  -- node [above] {$b$} ++(2,0) -- ++(0,2)
  -- node [above] {$a$} ++(-2,0)
  -- ++(-2,0) -- ++(0,-2)
  -- ++(-2,0) -- ++(0,-2)
  -- node [above] {$b$}  ++(-2,0) -- ++(0,-2)
  -- cycle;
\foreach \x in {0,2,4,6}
  {\draw [very thin,fill=white] (\x,\x) circle (2.2pt);
  \draw [fill=black!70] (\x+2,\x) circle (2pt);}
  ;
\foreach \x in {0,2,4}
  {\draw [very thin,fill=white] (\x+4,\x) circle (2.2pt);
  \draw [fill=black!70] (\x,\x+2) circle (2pt);}
  ;
\end{tikzpicture}\qquad\qquad\qquad\qquad
\begin{tikzpicture}[scale=2]
\fill[black!20] (0,0) -- ( 1.00000,  0.00000) -- ( 0.67844,  0.15485);
\fill[black!20] (0,0) -- ( 0.90096,  0.43388) -- ( 0.54407,  0.43388);
\fill[black!20] (0,0) -- ( 0.62348,  0.78183) -- ( 0.30193,  0.62698);
\fill[black!20] (0,0) -- ( 0.22252,  0.97492) -- ( 0.00000,  0.69589);
\fill[black!20] (0,0) -- (-0.22252,  0.97492) -- (-0.30193,  0.62698);
\fill[black!20] (0,0) -- (-0.62348,  0.78183) -- (-0.54407,  0.43388);
\fill[black!20] (0,0) -- (-0.90096,  0.43388) -- (-0.67844,  0.15485);
\fill[black!20] (0,0) -- (-1.00000,  0.00000) -- (-0.67844, -0.15485);
\fill[black!20] (0,0) -- (-0.90096, -0.43388) -- (-0.54407, -0.43388);
\fill[black!20] (0,0) -- (-0.62348, -0.78183) -- (-0.30193, -0.62698);
\fill[black!20] (0,0) -- (-0.22252, -0.97492) -- ( 0.00000, -0.69589);
\fill[black!20] (0,0) -- ( 0.22252, -0.97492) -- ( 0.30193, -0.62698);
\fill[black!20] (0,0) -- ( 0.62348, -0.78183) -- ( 0.54407, -0.43388);
\fill[black!20] (0,0) -- ( 0.90096, -0.43388) -- ( 0.67844, -0.15485);
\draw[very thin] 
  ( 1.00000,  0.00000) -- ( 0.67844,  0.15485) --
  ( 0.90096,  0.43388) -- ( 0.54407,  0.43388) --
  ( 0.62348,  0.78183) -- ( 0.30193,  0.62698) --
  ( 0.22252,  0.97492) -- ( 0.00000,  0.69589) --
  (-0.22252,  0.97492) -- (-0.30193,  0.62698) --
  (-0.62348,  0.78183) -- (-0.54407,  0.43388) --
  (-0.90096,  0.43388) -- (-0.67844,  0.15485) --
  (-1.00000,  0.00000) -- (-0.67844, -0.15485) --
  (-0.90096, -0.43388) -- (-0.54407, -0.43388) --
  (-0.62348, -0.78183) -- (-0.30193, -0.62698) --
  (-0.22252, -0.97492) -- ( 0.00000, -0.69589) --
  ( 0.22252, -0.97492) -- ( 0.30193, -0.62698) --
  ( 0.62348, -0.78183) -- ( 0.54407, -0.43388) --
  ( 0.90096, -0.43388) -- ( 0.67844, -0.15485) --
  cycle;
\end{tikzpicture}

\caption{Left: Escher stairs with $g$ steps, odd $g$ (here $g=3$).
Right: Unfolding of Ward's 
{\footnotesize$\displaystyle\Bigl(\frac\pi{14},\frac\pi7,\frac{11\,\pi}{14}\Bigr)$}
triangle.}\name{fig: Escher}

\end{figure}
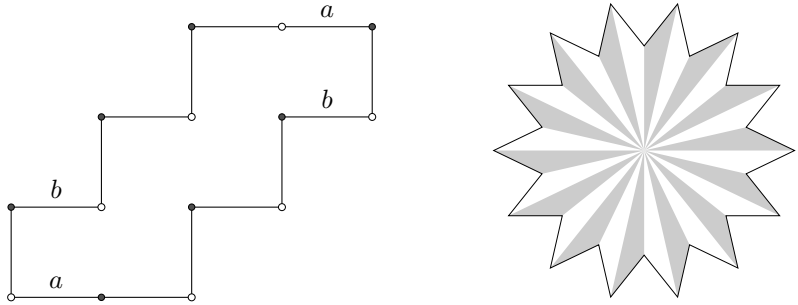

\begin{prop} \name{prop: Escher}
The Escher staircases have no convex presentations. 
\end{prop}

\begin{proof}
One shows using Schmith\"usen's algorithm \cite{Gabi} that there are two cusps,
corresponding to the horizontal and unit slope directions in Figure
\ref{fig: Escher}. These directions have no simple cylinders. 
\end{proof}

Ward \cite{Ward} constructed translation surfaces by applying the
`unfolding procedure' to the 
billiards in the $\displaystyle{\left(\frac{\pi}{2n}, \frac{\pi}{n},
    \frac{(2n - 3)\pi}{2n}\right)}$ triangles, and 
proved that the corresponding surfaces are lattice surfaces. See
Figure \ref{fig: Escher}.

\begin{prop}\name{prop: Ward}
The Ward examples have no convex presentations. 
\end{prop}
\begin{proof}
As proved by Ward, in his examples the Veech group has one
cusp. Therefore the horizontal direction in Figure \ref{fig: Escher}
represents the unique completely periodic direction (up to the
$\Gamma$-action). To locate the cylinders in this direction one simply
draws all horizontal segments starting at vertices of the polygon in
the figure; this analysis reveals that there are no
simple cylinders in the horizontal direction. Therefore the Ward
surfaces have no simple cylinders at all. 
\end{proof}

Bouw and M\"oller \cite{BM} found an infinite sequence of lattice
surfaces $\mathrm{BM}_{m,n}$. 

\begin{prop}\name{prop: BM}
When $n \geq m \geq 4$ and $n,m$ are not both even, the surface
$\mathrm{BM}_{m,n}$ has no convex presentation. 
\end{prop}
\begin{proof}
 Hooper \cite{Hooper BM} described the Bouw-M\"oller examples in terms
 of a decomposition into semi-regular polygons; see also
 \cite{wright}, where it is shown that Hooper's surfaces in fact
 coincide with $\mathrm{BM}_{m,n}$. The work of \cite{BM, Hooper BM,
   wright} also showed that the Veech group of
 $\mathrm{BM}_{m,n}$ is either the $(m,n,\infty)$ triangle group (when
 $n >m$), or the $(2,n,\infty)$-triangle group (when $n=m$ are odd), and hence
 these examples have a unique (up to the Veech group action) cylinder
 decomposition pattern. Hence any simple cylinder, if it
 existed, would have to be horizontal in a direction of a semi-regular
 polygon decomposition. It can be checked from Hooper's
 description (see Figure \ref{fig: BM}) that when $n \geq m \geq 4$ and $m,n$ are not both even,
 the surfaces $\mathrm{BM}_{m,n}$ 
 have no simple cylinders. 
\end{proof}

\begin{figure}[h]
\includegraphics{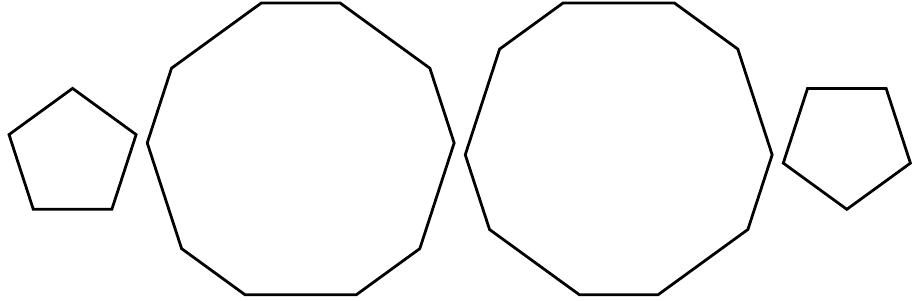}
\caption{Hooper's presentation of the Bouw-M\"oller surface
  $\mathrm{BM}_{4,5}$. Each edge is identified with an edge in a
  neighboring polygon.}
\label{fig: BM} 
\end{figure}

\begin{remark}
The property of having no simple cylinder is also $G$-invariant and
closed. It would be interesting to classify all surfaces without
simple cylinders. A similar closed and $G$-invariant property is
having no {\em semisimple cylinder}, i.e. a cylinder which is bounded on at least one side by a single
saddle connection.  
\end{remark}

\subsection{Some square-tiled examples}\name{subsec: square zoo}
We conclude this zoological survey with four interesting square-tiled
surfaces. No menagerie is complete without the legendary  {\em
  eierlegende Wollmilchsau} and {\em ornithorhynque} (see \cite{matheus
  blog}). One of  the many surprising features of these notoriously eccentric  
 square tiled surfaces, is the lack of simple cylinders. In
particular they have no convex presentations. 

\begin{prop}\name{prop: interesting square tiled}
The surface made of four
squares (see the left hand side of Figure \ref{fig: lots of squares}) has no convex
presentation.
Any square-tiled surface in $\HH(1,1)$ made of seven squares or fewer (see
 the two surfaces on the  right hand side of Figure \ref{fig: lots of squares}) has
no strictly convex presentations.
\end{prop}

\begin{figure}[!h]
\begin{tikzpicture}
\begin{scope}[xshift=-3.5cm,yshift=-1.25cm,rotate around={90:(1.5,1)}] 
\tikzstyle{every circle node}=[draw,inner sep=1pt,circle]
\draw (1,0) -- (1,1) -- (2,1) -- (2,2);
\draw
(0,0) node [circle,semithick,fill=white] {} -- node{\rotatebox{90}{\tiny$/$}}
(1,0) node [circle,very thin,fill=black] {} -- node{\rotatebox{90}{\tiny$/\!/$}}
(2,0) node [circle,semithick,fill=white] {} -- node{\rotatebox{90}{\scriptsize$=$}}
(2,1) node [circle,semithick,fill=white] {} -- node{\rotatebox{90}{\tiny$/\!/\!/$}}
(3,1) node [circle,very thin,fill=black] {} -- node{\rotatebox{90}{\scriptsize$-$}}
(3,2) node [circle,semithick,fill=white] {} -- node{\rotatebox{90}{\tiny$/\!/$}}
(2,2) node [circle,very thin,fill=black] {} -- node{\rotatebox{90}{\tiny$/\!/\!/$}}
(1,2) node [circle,semithick,fill=white] {} -- node{\rotatebox{90}{\scriptsize$-$}}
(1,1) node [circle,very thin,fill=black] {} -- node{\rotatebox{90}{\tiny$/$}}
(0,1) node [circle,semithick,fill=white] {} -- node{\rotatebox{90}{\scriptsize$=$}}
(0,0) -- cycle;
\end{scope}
\begin{scope}[xshift=0cm,yshift=0.25cm] 
\draw (0,0) node {$\circ$} -- node{\tiny$/\!/\!/\!/$}
(2,0) node {\textbullet} -- node{\tiny$/\!/\!/$}
(4,0) node {$\circ$} -- node{\tiny$/\!/$}
(6,0) node {\textbullet} -- node{\tiny$/$}
(7,0) node {$\circ$} -- node{$-$}
(7,1) node {\textbullet} -- node{\tiny$/\!/\!/\!/$}
(5,1) node {$\circ$}-- node{\tiny$/\!/\!/$}
(3,1) node {\textbullet}-- node{\tiny$/\!/$}
(1,1) node {$\circ$}-- node{\tiny$/$}
(0,1) node {\textbullet}-- node{$-$} (0,0) -- cycle;
\draw (1,0) -- (1,1) (2,0) -- (2,1) (3,0) -- (3,1)
  (4,0) -- (4,1) (5,0) -- (5,1) (6,0) -- (6,1);
\end{scope}
\begin{scope}[xshift=0cm,yshift=-1.75cm] 
\draw (0,0) node {$\circ$} -- node{\tiny$\backslash\!\backslash\!\backslash\!\backslash$} 
(3,0) node {\textbullet} -- node{\tiny$\backslash\!\backslash\!\backslash$}
(5,0) node {$\circ$} -- node{\tiny$\backslash\!\backslash$}
(6,0) node {\textbullet} -- node{\tiny$\backslash$}
(7,0) node {$\circ$} -- node{$=$}
(7,1) node {\textbullet} -- node{\tiny$\backslash\!\backslash\!\backslash\!\backslash$}
(4,1) node {$\circ$}-- node{\tiny$\backslash\!\backslash\!\backslash$}
(2,1) node {\textbullet}-- node{\tiny$\backslash\!\backslash$}
(1,1) node {$\circ$}-- node{\tiny$\backslash$}
(0,1) node {\textbullet}-- node{$=$} (0,0) -- cycle;
\draw (1,0) -- (1,1) (2,0) -- (2,1) (3,0) -- (3,1)
  (4,0) -- (4,1) (5,0) -- (5,1) (6,0) -- (6,1);
\end{scope}
\end{tikzpicture}
\caption{Left:\ this $4$-square-tiled surface in $\EE_{16}$ has no one-cylinder
direction and no convex presentation.\newline
Right:\ representatives for the two $\SL(2,\R)$-orbits of $7$-square-tiled
surfaces in $\EE_{49}$.  Both have no strictly convex presentation.}\name{fig: lots of squares}
\end{figure}

\begin{proof}
For the first assertion, checking the $\SL_2(\Z)$-orbit (see
\cite{Gabi}) one
finds that the $G$-orbit of this surface has two cusps, none of whose
corresponding cylinder decompositions is a  
one-cylinder decomposition. Only one of the cusps has a simple
cylinder $C$,
and one can verify that the foliation in the direction of any diagonal of
$C$ has leaves which do not cross $C$ and are not on saddle connections. So by Proposition \ref{prop:
  central symmetry}, the
surface has no convex presentation. 

For the second assertion, let $M \in \HH(1,1)$ be square-tiled, and
suppose $C$ is a simple 
cylinder of a convex presentation, as in Proposition \ref{prop:
  central symmetry}. We claim that $C$ has area at least four (where we
normalize area so that each square is of unit area). Indeed,
replacing $M$ by its image under
an element of $\SL_2(\Z)$ we can represent $C$ as a rectangle in the plane with
corners at integer points, where all vertices of $M$ remain at integer
points, such that  the two boundary edges of $C$ are horizontal. In
this situation, a generalized diagonal is an edge crossing $C$ from top to
bottom, where it is allowed to wrap around the circumference of
$C$. For rectangles of area less than four, straight lines starting at
integer points in the plane, in the directions of
generalized diagonals, can intersect the boundary of $C$ in at most
three points. However according to Proposition \ref{prop: central
  symmetry}, there must be four intersection points, corresponding to
the discontinuity points of the interval exchange. This proves the
claim. 

Now let $P$ be a convex polygon representing $M$. Since $M \in
\HH(1,1)$, $P$ is a decagon, and $P \sm C$ is a
union of two quadrilaterals. Identifying their edges which are
boundary components of $C$, we obtain a new strictly convex polygon $P'$. The
vertices of $P'$ are integer points and $P'$ is an octagon of area at
most three. 
Recall Pick's formula \cite{pick}, which says that the area of $P'$ is
$i+b/2-1$, where $i$ is the number of interior integer points and $b$
is the number of integer points along the edges. By construction $b
\geq 8$, so in fact $b=8$, $i=0$ and the area of $P'$ is exactly three.
Label the sides of $P'$ by
vectors $v_1, \ldots, v_8$, cyclically in the counterclockwise direction. Applying $\SL_2(\Z)$ again
we can assume that $v_1, v_2$ meet at the origin and lie respectively
along the $y$-axis and $x$-axis. Since the eight vertices are the only
integer points on sides, $v_1, v_2$ are of unit length. By strict
convexity, the internal angle between $v_8$ and $v_1$ is more than
$\pi/2$ and so is the internal angle between $v_2$ and $v_3$. This
forces $(1,1)$ to be an interior point of $P'$, a contradiction. 
\end{proof}

\section{A summary of results on genus two surfaces}\name{sec:
genus 2}

In this section we summarize results about genus two surfaces
which will be crucial for our discussion. These results incorporate
work of many authors, notably McMullen. We refer the reader
to the surveys
\cite{ handbook HS, zorich survey} for a more detailed discussion.  

All surfaces considered in
this section will be of genus two. 
Given a positive
integer $D$ let
$\mathcal{O}_D$ be the order of discriminant $D$ in the field
$\Q\left(\sqrt{D}\right)$. We say that $M$ is an
{\em eigenform surface of discriminant $D$} if $\mathcal{O}_D$ embeds
in the endomorphism ring of  its Jacobian. 

\begin{thm}[McMullen]\name{thm: McMullen dynamics H(2)}
Any $G$-orbit in  $\HH(2)$ is either closed or dense. 
\end{thm}


\begin{thm}[Veech, Calta, McMullen]\name{thm: Calta McMullen H(2)}
The following are equivalent for $M \in \HH(2)$:
\begin{itemize}
\item
 $M$ is a lattice surface. 
\item
 $M$ is an eigenform
surface. 
\item
Any saddle  connection direction is completely periodic. 
\item
Any cylinder direction is completely periodic. 
\end{itemize}
Suppose the above hold and let $D$ be the corresponding
discriminant. Then: 
\begin{itemize}

\item
$M$ is arithmetic if and only if $D$ is a square.
\item
The number of $G$-orbits of lattice surfaces of discriminant $D \geq 5$ is:
\eq{eq: discriminant numbers}{
\begin{split}
0 & \ \ \mathrm{if } \ D \equiv 2,3 \mod 4. \\
1 & \ \ \mathrm{if } \ D \equiv 0, 4,5 \mod 8. \\
2  & \ \ \mathrm{if} \ D \equiv 1 \mod 8.
\end{split}
}
\end{itemize}
\end{thm}
We will refer to the $G$-orbits of lattice surfaces in $\HH(2)$ by their
discriminant $D$; in case $D \equiv 1 \mod 8$ we will write $D_0, D_1$
to distinguish the two cases in \equ{eq: discriminant numbers} (these are distinguished by the parity of
their spin structure, see \cite{McMullen disc}). We will abuse
notation by referring to the symbols $D_0, D_1$ as discriminants.

\begin{figure}[!h]
\begin{tikzpicture}
\begin{scope}[xshift=0cm,yshift=0cm] 
\draw [<->,shorten >=0.1cm,shorten <=0.1cm,xshift=-0.2cm] (0,0) -- node [left] {$c$} (0,1);
\draw [<->,shorten >=0.1cm,shorten <=0.1cm,xshift=-0.2cm] (0,0) -- node [left] {$\lambda$} (0,-1.3028);
\draw [<->,shorten >=0.1cm,shorten <=0.1cm,yshift=0.2cm] (1,1) -- node [above] {$b$} (4,1);
\draw [<->,shorten >=0.1cm,shorten <=0.1cm,yshift=0.2cm] (0,1) -- node [above] {$a$} (1,1);
\draw [<->,shorten >=0.1cm,shorten <=0.1cm,yshift=-0.2cm] (0,-1.3028) -- node [below] {$\lambda$} (1.3028,-1.3028);
\draw [->] (1,1) -- node [right] {$v_1$} (1,0.3);
\draw [->] (2.3028,1) -- node [right] {$v_2$} (2.3028,0.3);
\draw
(0,0) node {\scriptsize\textbullet} -- node {\tiny$-$} 
(0,-1.3028) node {\scriptsize\textbullet} -- node {\tiny$/$}
(1.3028,-1.3028) node {\scriptsize\textbullet} -- node {\tiny$-$}
(1.3028,0) node {\scriptsize\textbullet} -- node {\tiny$/\!/$}
(3,0) node {\scriptsize\textbullet} -- node {\tiny$=$}
(4,1) node {\scriptsize\textbullet} -- node {\tiny$/\!/$}
(2.3028,1) node {\scriptsize\textbullet} -- node {\tiny$/$}
(1,1) node {\scriptsize\textbullet} -- node {\tiny$=$}
(0,0) -- cycle;
\end{scope}
\end{tikzpicture}
\caption{Dimensions and labels for a prototype in $\HH(2)$.}\name{fig:
  prototype}
\end{figure}
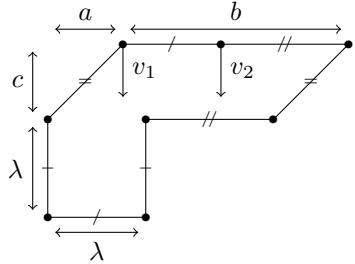

\begin{thm}[Calta-McMullen prototypes]\name{thm: McMullen prototypes H(2)}
Suppose $\theta$ is a completely periodic direction on a lattice
surface in $\HH(2)$ of discriminant $D$, such that the corresponding
cylinder decomposition has two cylinders. Then one of the cylinders $C$ is
simple, and there are integers $a,
b, c, e$ such that:
\eq{eq: prototypes}{\gcd(a,b,c,e)=1,  \ D=e^2+4bc,  \ b>0,  \ c>0,
\  c+e<b, \ 0 \leq a <
b.}
After applying an affine map which makes the core curve of $C$ 
horizontal, a saddle connection crossing $C$ vertical, and the height
of $C$ equal to its circumference, and equal to $\displaystyle{\lambda
  = \frac{e + \sqrt{D}}{2}},$ the surface is presented as in Figure
\ref{fig: prototype}. 

Furthermore, for each $D$ there are only finitely many choices of
$a,b,c,e$ satisfying \equ{eq: prototypes}, and for each such choice, the
surface of Figure \ref{fig: prototype} is a surface of discriminant
$D$ with a completely periodic horizontal direction. In case $D \equiv
1 \mod 8$, the parity of its spin structure is 
\eq{eq: detecting spin}{
\frac{e-f}{2} + (c+1)(a+b+ab) \mod 2,
}
where $f$ depends on $D$.  
\end{thm}

\begin{figure}[!h]
\resizebox{\textwidth}{!}{%
\begin{tikzpicture}
\begin{scope}[xshift=8cm,yshift=-5cm] 
\draw (0,0) node {\scriptsize\textbullet} -- node {\tiny$-$} 
(0,-0.3028) node {\scriptsize\textbullet} -- node {\tiny$/$}
(0.3028,-0.3028) node {\scriptsize\textbullet} -- node {\tiny$-$}
(0.3028,0) node {\scriptsize\textbullet} -- node {\tiny$/\!/$}
(1,0) node {\scriptsize\textbullet} -- node {\tiny$=$}
(1,1) node {\scriptsize\textbullet} -- node {\tiny$/\!/$}
(0.3028,1) node {\scriptsize\textbullet} -- node {\tiny$/$}
(0,1) node {\scriptsize\textbullet}-- node {\tiny$=$}
(0,0) -- cycle;
\end{scope}
\begin{scope}[xshift=-3.75cm,yshift=0cm] 
\draw (0,0) node {\scriptsize\textbullet} -- node {\tiny$-$} 
(0,-1.3028) node {\scriptsize\textbullet} -- node {\tiny$/$}
(1.3028,-1.3028) node {\scriptsize\textbullet} -- node {\tiny$-$}
(1.3028,0) node {\scriptsize\textbullet} -- node {\tiny$/\!/$}
(3,0) node {\scriptsize\textbullet} -- node {\tiny$=$}
(3,1) node {\scriptsize\textbullet} -- node {\tiny$/\!/$}
(1.3028,1) node {\scriptsize\textbullet} -- node {\tiny$/$}
(0,1) node {\scriptsize\textbullet}-- node {\tiny$=$}
(0,0) -- cycle;
\end{scope}
\begin{scope}[xshift=-3.75cm,yshift=-3cm] 
\draw (0,0) node {\scriptsize\textbullet} -- node {\tiny$-$} 
(0,-2.3028) node {\scriptsize\textbullet} -- node {\tiny$/$}
(2.3028,-2.3028) node {\scriptsize\textbullet} -- node {\tiny$-$}
(2.3028,0) node {\scriptsize\textbullet} -- node {\tiny$/\!/$}
(3,0) node {\scriptsize\textbullet} -- node {\tiny$=$}
(3,1) node {\scriptsize\textbullet} -- node {\tiny$/\!/$}
(2.3028,1) node {\scriptsize\textbullet} -- node {\tiny$/$}
(0,1) node {\scriptsize\textbullet}-- node {\tiny$=$}
(0,0) -- cycle;
\end{scope}
\begin{scope}[xshift=0cm,yshift=0cm] 
\draw (0,0) node {\scriptsize\textbullet} -- node {\tiny$-$} 
(0,-1.3028) node {\scriptsize\textbullet} -- node {\tiny$/$}
(1.3028,-1.3028) node {\scriptsize\textbullet} -- node {\tiny$-$}
(1.3028,0) node {\scriptsize\textbullet} -- node {\tiny$/\!/$}
(3,0) node {\scriptsize\textbullet} -- node {\tiny$=$}
(4,1) node {\scriptsize\textbullet} -- node {\tiny$/\!/$}
(2.3028,1) node {\scriptsize\textbullet} -- node {\tiny$/$}
(1,1) node {\scriptsize\textbullet}-- node {\tiny$=$}
(0,0) -- cycle;
\end{scope}
\begin{scope}[xshift=4cm,yshift=0cm] 
\draw (0,0) node {\scriptsize\textbullet} -- node {\tiny$-$} 
(0,-1.3028) node {\scriptsize\textbullet} -- node {\tiny$/$}
(1.3028,-1.3028) node {\scriptsize\textbullet} -- node {\tiny$-$}
(1.3028,0) node {\scriptsize\textbullet} -- node {\tiny$/\!/$}
(3,0) node {\scriptsize\textbullet} -- node {\tiny$=$}
(5,1) node {\scriptsize\textbullet} -- node {\tiny$/\!/$}
(3.3028,1) node {\scriptsize\textbullet} -- node {\tiny$/$}
(2,1) node {\scriptsize\textbullet}-- node {\tiny$=$}
(0,0) -- cycle;
\end{scope}
\begin{scope}[xshift=0cm,yshift=-3cm] 
\draw (0,0) node {\scriptsize\textbullet} -- node {\tiny$-$} 
(0,-2.3028) node {\scriptsize\textbullet} -- node {\tiny$/$}
(2.3028,-2.3028) node {\scriptsize\textbullet} -- node {\tiny$-$}
(2.3028,0) node {\scriptsize\textbullet} -- node {\tiny$/\!/$}
(3,0) node {\scriptsize\textbullet} -- node {\tiny$=$}
(4,1) node {\scriptsize\textbullet} -- node {\tiny$/\!/$}
(3.3028,1) node {\scriptsize\textbullet} -- node {\tiny$/$}
(1,1) node {\scriptsize\textbullet}-- node {\tiny$=$}
(0,0) -- cycle;
\end{scope}
\begin{scope}[xshift=4cm,yshift=-3cm] 
\draw (0,0) node {\scriptsize\textbullet} -- node {\tiny$-$} 
(0,-2.3028) node {\scriptsize\textbullet} -- node {\tiny$/$}
(2.3028,-2.3028) node {\scriptsize\textbullet} -- node {\tiny$-$}
(2.3028,0) node {\scriptsize\textbullet} -- node {\tiny$/\!/$}
(3,0) node {\scriptsize\textbullet} -- node {\tiny$=$}
(5,1) node {\scriptsize\textbullet} -- node {\tiny$/\!/$}
(4.3028,1) node {\scriptsize\textbullet} -- node {\tiny$/$}
(2,1) node {\scriptsize\textbullet}-- node {\tiny$=$}
(0,0) -- cycle;
\end{scope}
\end{tikzpicture}}
\caption{The seven prototypes for discriminant $13$ in $\HH(2)$.}
\end{figure}
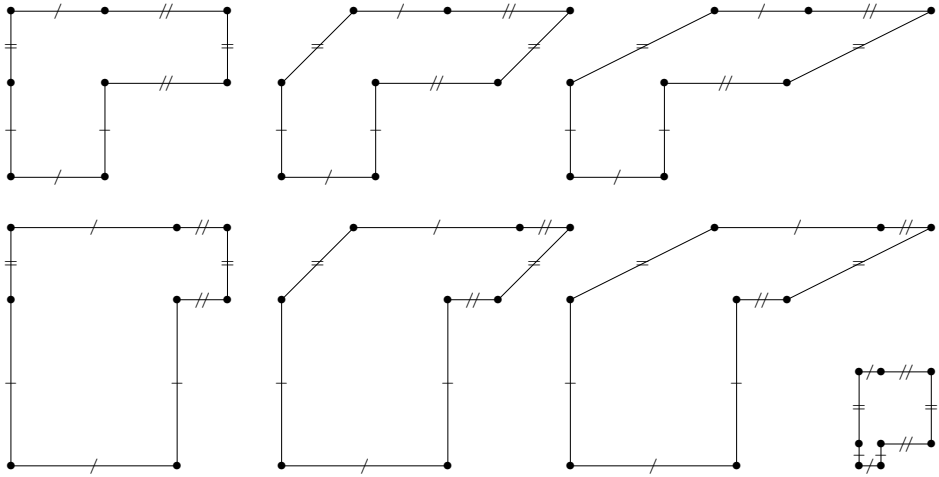

\begin{remark}
Our notation differs slightly from that of McMullen \cite{McMullen
  disc}, who took $a< \gcd(b,c)$ and introduced an additional
parameter $t$. Our $a$ corresponds to $a+tc$ in McMullen's
notation. McMullen's notation was better suited to his purpose of
listing the
different prototypes which correspond to the 
same cusp of the corresponding $G$-orbit -- information
which will not be relevant to our discussion. 
\end{remark}

Let $\EE_D$ be the {\em eigenform locus of discriminant $D$},
i.e. the subset of $\HH(1,1)$ consisting of eigenform
surfaces of discriminant $D$. We say that two surfaces $M_1, M_2 \in
\HH(1,1)$ are {\em rel equivalent} if there is an identification
of their homology groups $H_1(M_i ; \Z)$ such that corresponding
1-cycles have the same holonomy. This equivalence relation is
generated by a local surgery which moves one singularity with respect
to the other without changing absolute periods. See \S \ref{sec:
  unstable convexity} for more
details. 

\begin{thm}[Calta, McMullen]\name{thm: Calta McMullen H(1,1)}
The eigenform locus $\EE_D$ is closed and invariant under both the $G$-action
and rel-equivalence. The following are
equivalent for $M \in \HH(1,1)$:
\begin{itemize}
\item
$M$ is an eigenform surface. 
\item
Any cylinder direction is completely periodic. 
\end{itemize}

\end{thm}

\begin{thm}[McMullen]\name{thm: McMullen eigenform} For $D \geq 4, D
  \equiv 0,1 \mod 4$, 
each $\EE_D$ is a connected 5-dimensional suborbifold of $\HH(1,1)$ which is
equipped with a unique smooth $G$-invariant ergodic measure. This measure is
finite. 
\end{thm}

\begin{thm}[McMullen]\name{thm: McMullen dynamics H(1,1)}
For any surface in $M \in \HH(1,1)$, one of the following (mutually
exclusive) assertions holds:
\begin{itemize}
\item
$M$ is an eigenform surface of discriminant $D$ and is a lattice surface. In this case $M$
is either arithmetic, and $D$ is a square, or $M$ 
is in the orbit of the
regular decagon, and $D=5$. 
\item
$M$ is an eigenform surface of discriminant $D$ and its $G$-orbit is
dense in $\EE_D$. 
\item
$M$ is not an eigenform surface and its $G$-orbit is dense in $\HH(1,1)$. 
\end{itemize}
\end{thm}

Given $x \in \R$ and $b>0$ we write 
\eq{eq: defn Rb}{R_b(x) \df x - \left \lfloor
\frac{x}{b} \right\rfloor \cdot b,}
that is $R_b(x) $ is the unique element
in $x + b \Z \cap \left[0, b \right
).$ The following result is proved in \cite{McMullen disc, Bainbridge}

\begin{thm}[McMullen, Bainbridge]\name{thm: prototypes H(1,1)}
Suppose $\theta$ is a completely periodic direction on an eigenform 
surface of discriminant $D$ in $\HH(1,1)$, such that the corresponding
cylinder decomposition has three cylinders, or two cylinders with one
simple cylinders. Then there are integers $a,
b, c, e$, and real numbers $x,y$ such that $a,b,c,e$
satisfy \equ{eq: prototypes}, and 
\eq{eq: x and y}{
0 \leq x < b-\lambda, \ \ 0 \leq y < \displaystyle{\frac{\lambda c}{\lambda+c}}. 
}
After applying an affine map as in Theorem \ref{thm: McMullen
  prototypes H(2)}, and setting 
\eq{eq: s t defined}{
s \df R_b \left( a + x\left(1 + \frac{c}{\lambda} \right) \right), \ t \df c-y \left(1+\frac{c}{\lambda}  \right)
}
the surface is presented as in Figure
\ref{fig: prototype H(1,1)}. The number of cylinders in the cylinder
decomposition is two if and only if $y=0$. 

Furthermore, suppose $a,b,c,e,x,y$ satisfy \equ{eq: prototypes} and
\equ{eq: x and y}, and define $s, t$ via \equ{eq: s t defined}. Then the
surface of Figure \ref{fig: prototype} is a surface of discriminant
$D$ with a completely periodic horizontal direction.
\end{thm}

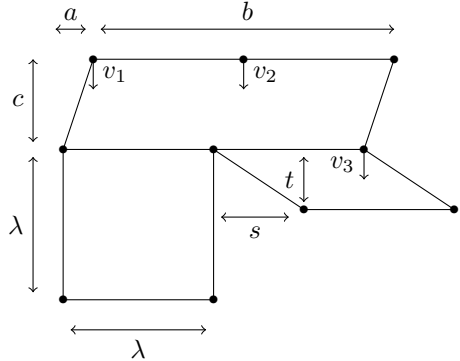
\begin{figure}[!h]
\begin{tikzpicture}
\begin{scope}[xshift=0cm,yshift=0cm,scale=2] 
\coordinate (A) at (0, 0);
\coordinate (H) at (0.2, 0.6);
\coordinate (G) at (1.2, 0.6);
\coordinate (F) at (2.2, 0.6);
\coordinate (E) at (2, 0);
\coordinate (B) at (0, -1);
\coordinate (C) at (1, -1);
\coordinate (D) at (1, 0);
\coordinate (I) at (1.6, -0.4);
\coordinate (J) at (2.6, -0.4);
\begin{scope}
\path[clip] (A) -- (B) -- (C) -- (D) -- (I) -- (J) -- (E) -- (F) -- (H) -- 
(A) -- cycle;
\draw [->] (G) -- node [right] {$v_2$} +(0,-0.2);
\draw [->] (E) -- node [left] {$v_3$} +(0,-0.2);
\draw [->] (H) -- node [right] {$v_1$} +(0,-0.2);
\end{scope}
\draw [<->,shorten >=0.1cm,shorten <=0.1cm,xshift=-0.2cm] (0,0) --
node [left] {$c$} (0,0.65);
\draw [<->,shorten >=0.1cm,shorten <=0.1cm,xshift=-0.2cm] (0,0) -- node [left] {$\lambda$} (0,-1);
\draw [<->,shorten >=0.1cm,shorten <=0.1cm,yshift=0.2cm] (0.2,0.6) -- node [above] {$b$} (2.25,0.6);
\draw [<->,shorten >=0.1cm,shorten <=0.1cm,yshift=0.2cm] (-0.1,0.6) -- node [above] {$a$} (0.2,0.6);
\draw [<->,shorten >=0.1cm,shorten <=0.1cm,yshift=-0.2cm] (0,-1) --
node [below] {$\lambda$} (1,-1);
\draw [<->,shorten >=0.1cm,shorten <=0.1cm,yshift=-0cm] (1,-0.45) -- node [below] {$s$} (1.55,-0.45);
\draw [<->,shorten >=0.1cm,shorten <=0.1cm,xshift=0.5cm] (1.1,0) -- node [left] {$t$} (1.1,-0.4);
\draw
(A) node {\scriptsize\textbullet} -- (H) node {\scriptsize\textbullet} -- 
(G) node {\scriptsize\textbullet} -- (F) node {\scriptsize\textbullet} --
(E) node {\scriptsize\textbullet} -- (A) -- (B) node {\scriptsize\textbullet}
-- (C) node {\scriptsize\textbullet} -- (D) node {\scriptsize\textbullet} --
(I) node {\scriptsize\textbullet} -- (J) node {\scriptsize\textbullet} -- (E);
\end{scope}
\end{tikzpicture}
\caption{Dimensions and labels for a prototype 
in $\HH(1,1)$.}\name{fig: prototype H(1,1)}
\end{figure}

\combarak{Check what is actually in the papers of McMullen and
  Bainbridge to figure out whether we need to include a proof. }
Note that when $y>0$ there are two simple cylinders in the presentation, the
taller one is represented by a $\lambda \times \lambda$ square. By an
abuse of notation which should cause no confusion, we
will refer to the latter simple cylinder as {\em the} simple cylinder
of the prototype. 

\subsection{One-cylinder presentations}
We say that $M$ has a {\em one-cylinder presentation} if it can be
presented as one cylinder with gluings along its boundary edges. 
The following useful result is well-known to those who know it well. For an
elegant proof of Hubert, see \cite[Lemma 5.3]{HLM}.  For a
generalization, see \cite{Calta Smillie}. 
\begin{prop}
If $M$ is a lattice surface with a one-cylinder presentation then $M$
is arithmetic. 
\end{prop}

We will also need the following statement, which follows from
\cite{McMullen disc, HL}:
\begin{prop}[Hubert-Leli\`evre, McMullen]\name{Hubert Lelievre McMullen}
Any arithmetic surface in $\HH(2)$ has a one-cylinder presentation. 
\end{prop}
\begin{remark}
The conclusion of Proposition \ref{Hubert Lelievre McMullen} fails in
$\HH(1,1)$. For an example, see the surface portrayed on the left hand
side of Figure
\ref{fig: lots of squares}.  
\end{remark}

\subsection{Finiteness result of Eskin-Mirzakhani-Mohammadi}
In recent breakthrough work \cite{EM, EMM}, severe restrictions were
found on the possible $G$-invariant measures and $G$-orbit-closures in
any stratum $\HH$. We will need the following:
\begin{thm}[see \cite{EMM}, Theorem 2.2]
\name{thm: finiteness EMM}
For any stratum $\HH$, any closed $G$-invariant subset is a union of
finitely many orbit-closures. 

\end{thm}

As an immediate consequence of Proposition \ref{prop: G invariant
  closed} and Theorems \ref{thm: McMullen dynamics H(2)}, \ref{thm:
  Calta McMullen H(2)}, \ref{thm: 
McMullen dynamics H(1,1)} and \ref{thm:
  finiteness EMM}, we have: 
\begin{cor}\name{cor: EMM finitely many}
In $\HH(2)$, the set of surfaces with no strictly convex presentations
consists of finitely many $G$-orbits of lattice surfaces. In
$\HH(1,1)$, the set of surfaces with no strictly convex presentations
consists of finitely many $G$-orbits of lattice surfaces and finitely
many  eigenform loci. 
\end{cor}
\qed

\begin{remark}
One could derive Corollary \ref{cor: EMM finitely many}  without
appealing to \cite{EMM}, using McMullen's results \cite{McMullen-SL(2)} and an
analysis of horocycle trajectories close to the supports of the
$G$-invariant measures of $\HH(1,1)$. We omit the details.  
\end{remark}

\section{The canonical polygon}
\name{sec: nec and suff}
In this section we will start with a surface of genus two presented
via the prototypes of \S\ref{sec: genus 2}, and give an alternative presentation
of the same surface via what we will call the {\em canonical
  polygon}. This will be an octagon (resp. a decagon) with edge
identifications giving rise to a surface in $\HH(2)$
(resp. $\HH(1,1)$).  The construction of the canonical polygon is  
explicit and algorithmic. We will show that if a surface of genus two admits a
strictly convex presentation, then it is a canonical polygon for
some prototype. We will discuss the two strata $\HH(2)$ and $\HH(1,1)$
separately. 

\subsection{Canonical polygons in $\HH(2)$}\name{subsec: canonical H(2)}
Let $M$ be a lattice surface in $\HH(2)$ presented via a prototype as 
 in
Theorem \ref{thm: McMullen prototypes H(2)}, and let $C$ be the simple
cylinder of the prototype. Recall that $C$ is
normalized so that its circumference and height are both equal to
$\lambda$, its circumference is horizontal, and there is a vertical
saddle connection passing through $C$ from bottom to top. We will say
that the prototype is 
{\em nondegenerate for $M$} if every non-critical vertical leaf passes
through $C$ infinitely many times, and every critical leaf passes
through the interior of 
$C$ at least once. 
Otherwise we will say the prototype is {\em degenerate for $M$}. 
Let $\sigma$ be the horizontal saddle connection which is
the top boundary segment of $C$.

\begin{prop}\name{prop: iet}
The prototype with parameters $a,b,c,e,
\lambda, D$ is nondegenerate for $M$
 if and only if there are non-negative integers $k$ and $\ell$ such
that the following hold:
\eq{eq: first prong}{
R_b(ka) \in \left(0, \lambda \right)  \ \mathrm{and} \ 
R_b(ja)\notin [0, \lambda] \ \mathrm{ for} \  j=1, \ldots, k-1,}
and 
\eq{eq: second prong}{
R_b(\lambda + \ell a)\in \left(0, \lambda \right) \ \mathrm{and
  }\ R_b(\lambda+ ja )\notin [0, \lambda]\ \mathrm{for} \  j=1, \ldots, \ell-1.}
In this case the return map to $\sigma$
going upward along vertical leaves is an interval exchange on three
intervals, and the corresponding permutation is the `inversive permutation' $i \mapsto 4-i.$ 

\end{prop}  

\begin{proof}Nondegeneracy is equivalent to all minimal components of
  the vertical flow intersecting the interior of $C$. Equivalently, all vertical
  separatrices going down from the singularity pass through
  $C$. Let $k$ and $\ell$ denote respectively the number of times
  the downward vertical separatrices marked $v_1, v_2$ in Figure
  \ref{fig: prototype}, 
  pass through the non-simple cylinder in the 
  prototype, before reaching the interior of $C$. The horizontal
  coordinate of a point changes by an addition of $a$ whenever
  the separatrix passes through the non-simple cylinder without
  reaching $C$. This leads to the formulae \equ{eq: first prong} and
\equ{eq: second prong}. 

The two points where the two prongs touch $\sigma$ are the two
discontinuities of the interval exchange which is the return map to
$\sigma$. Thus the 
  interval exchange has two discontinuities, i.e. three intervals of continuity. 
The center of $C$ is a Weierstrass point. Symmetry of the vertical
foliation with respect to the hyperelliptic involution implies that
the permutation is inversive, as claimed. 
\end{proof}

The canonical polygon for the prototype is constructed as follows (see
Figure \ref{fig: canonical polygon}): For each of the three intervals
of continuity of the interval exchange in Proposition \ref{prop: iet},
construct the strip going up along vertical leaves from the top of $C$ to the bottom of
$C$, and number them from left to right. The singularity of $M$
appears on the right-hand boundary edge of the first strip, on both
boundary edges of the second
strip, and on the left-hand boundary edge of the third strip. Cut the
strips by segments joining the singularity to itself, one in each
strip, and rearrange. The canonical polygon is the resulting octagon.

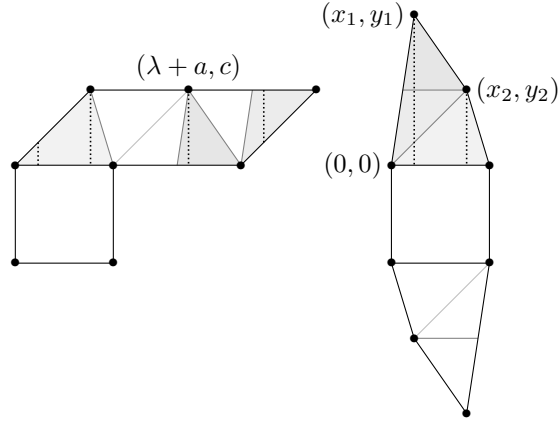
\begin{figure}[!h]
\begin{tikzpicture}
\begin{scope}[xshift=0cm,yshift=0cm] 
\coordinate (A) at (0,0);
\coordinate (B) at (0,-1.3028);
\coordinate (C) at (1.3028,-1.3028);
\coordinate (D) at (1.3028,0);
\coordinate (E) at (3,0);
\coordinate (F) at (4,1);
\coordinate (G) at (2.3028,1);
\node[above] at (G) {$(\lambda+a,c)$};
\coordinate (H) at (1,1);
\coordinate (GE) at (2.1514,0);
\coordinate (EG) at (3.1514,1);
\coordinate (G1) at (2.3028,0.);
\coordinate (G2) at (3.3028,1);
\coordinate (G3) at (3.3028,0);
\coordinate (G4) at (0.3028,1);
\coordinate (G5) at (0.3028,0);
\fill [gray!10] (A) -- (D) -- (H) -- cycle;
\fill [gray!20] (G) -- (GE) -- (E) -- cycle;
\fill [gray!15] (E) -- (EG) -- (F) -- cycle;
\draw [gray] (H) -- (D);
\draw [gray!50] (D) -- (G);
\draw [gray] (E) -- (EG);
\draw [gray] (GE) -- (G);
\draw [gray] (G) -- (E);
\draw (A) -- (D);
\draw [semithick, densely dotted] (H) -- (1,0);
\draw [semithick, densely dotted] (G) -- (G1);
\draw [semithick, densely dotted] (G2) -- (intersection of E--F and G2--G3);
\draw [semithick, densely dotted] (intersection of A--H and G4--G5) -- (G5);
\draw (A) node {\scriptsize\textbullet}
--
(B) node {\scriptsize\textbullet}
--
(C) node {\scriptsize\textbullet}
--
(D) node {\scriptsize\textbullet}
--
(E) node {\scriptsize\textbullet}
--
(F) node {\scriptsize\textbullet}
--
(G) node {\scriptsize\textbullet}
--
(H) node {\scriptsize\textbullet}
--
(0,0) -- cycle;
\end{scope}
\begin{scope}[xshift=5cm,yshift=0cm] 
\coordinate (A) at (0,0);
\node[left] at (A) {$(0,0)$};
\coordinate (B) at (0,-1.3028);
\coordinate (C) at (1.3028,-1.3028);
\coordinate (D) at (1.3028,0);
\coordinate (G) at (0.3028,2);
\node[left] at (G) {$(x_1,y_1)$};
\coordinate (H) at (1,1);
\node[right] at (H) {$(x_2,y_2)$};
\coordinate (HH) at (1,1);
\coordinate (HH) at (0.3028,-2.3028);
\coordinate (GG) at (1,-3.3028);
\coordinate (H1) at (0,1);
\coordinate (HH1) at (1.3028,-2.3028);
\coordinate (Hl) at (intersection of A--G and H1--H);
\fill [gray!10] (A) -- (D) -- (H) -- cycle;
\fill [gray!15] (A) -- (H) -- (Hl) -- cycle;
\fill [gray!20] (Hl) -- (H) -- (G) -- cycle;
\draw [gray] (A) -- (H);
\draw (B) -- (C);
\draw [gray] (intersection of A--G and H1--H) -- (H);
\draw [gray] (intersection of GG--C and HH1--HH) -- (HH);
\draw [gray!50] (C) -- (HH);
\draw (A) -- (D);
\draw [semithick, densely dotted] (H) -- (1,0);
\draw [semithick, densely dotted](G) -- (0.3028,0.);
\draw (A) node {\scriptsize\textbullet}
--
(B) node {\scriptsize\textbullet}
--
(HH) node {\scriptsize\textbullet}
-- (GG) node {\scriptsize\textbullet}
-- (C) node {\scriptsize\textbullet}
--
(D) node {\scriptsize\textbullet}
--
(H) node {\scriptsize\textbullet}
--
(G) node {\scriptsize\textbullet}
-- (0,0) -- cycle;
\end{scope}
\end{tikzpicture}
\caption{Flowdown lines and canonical polygon of a prototype:\
$(D,a,b,c,e)=(13,1,2,-1)$;
$\lambda=\frac{\sqrt{13}-1}{2}.$}\name{fig: canonical polygon}
\end{figure}

Let $(x_1, y_1), (x_2, y_2)$ denote respectively the holonomies of the
segment going from the top left-hand corner of $C$ to the singularity
of $M$, by following the top of $C$ until the first (respectively
second) discontinuity of the interval exchange transformation
described in Proposition \ref{prop: iet}, and then up to the
singularity along the right hand edge of the first (resp. second) strip.

\begin{prop}\name{prop: canonical octagon}
Let $k$ and $\ell$ be as in \equ{eq: first prong} and \equ{eq: second
  prong}. Then 
\eq{eq: formulae octagon}{
\left( \begin{matrix} x_1 \\ y_1\end{matrix}\right)
=\left( \begin{matrix} R_b(\lambda+\ell a) \\
\ell c  \end{matrix} \right), \ \ \left(\begin{matrix}x_2 \\ y_2
   \end{matrix}\right) = \left(\begin{matrix} R_b(k a) \\ kc \end{matrix}\right),
}
and $x_1 < x_2$. 
The canonical polygon is strictly convex if and only if 
\eq{eq: convexity octagon}{
x_2y_1 - x_1y_2>0, \ x_2y_1-x_1y_2 > \lambda(y_1-y_2).
}
\combarak{Double check that what I got here is what the computer is
  actually computing.}
\end{prop} 
\begin{proof}
The points $(x_i, y_i), \ i=1,2$ are vertices of the octagon in light
of Proposition \ref{prop: iet}. To find out which of these two points is to the
left of the other, we check how these vertices are identified
with the vertices in the prototype diagram; we see that in the
octagon, the vertical prong going down from the vertex at $(x_2, y_2)$
is reached from the vertical prong at $(0,0)$ 
by moving counterclockwise for an angle of $2\pi$. The same
description is valid for the vertical prong at $(\lambda+ a, c)$, thus
they represent the same prong. 

The two inequalities in \equ{eq: convexity octagon} correspond to internal angles less than
$\pi$ at each of the two vertices $(x_1, y_1), (x_2, y_2)$. The other
six vertices follow automatically from these via the hyperelliptic
involution, and from the inequalities $0<x_1<x_2 < \lambda$. 
\end{proof}
Our interest in the canonical polygon is motivated by the following:
\begin{prop}\name{prop: nec and suff}
If $M$ is a lattice surface in $\HH(2)$ and has a strictly convex
presentation, then there is a prototype which is nondegenerate for
$M$, for which the given convex polygon presentation coincides with
the canonical polygon. 
\end{prop}

\begin{proof}
Let $\mathcal{P}$ be the convex polygon representing $M$ and let $C$
be a simple cylinder as in Proposition \ref{prop: central 
  symmetry}. Then the direction of the circumference of $C$ is
completely periodic by Theorem \ref{thm: Calta McMullen
  H(2)}. Consider the prototype corresponding to this direction via
Theorem \ref{thm: McMullen prototypes H(2)}. The convexity of
$\mathcal{P}$ implies that the 
prototype is nondegenerate for $M$. The definition of $(x_i, y_i)$
given in the discussion preceding Proposition \ref{prop: canonical
  octagon} now shows that $\mathcal{P}$ coincides with the canonical
polygon for this prototype. 
\end{proof}

\subsection{Canonical polygons in $\HH(1,1)$}\name{subsec: canonical H(1,1)}
We proceed analogously to \S \ref{subsec: canonical H(2)}, making the
required adjustments. 
Let $M$ be an eigenform surface in $\HH(2)$. Suppose $M$ is presented via a prototype as 
 in
Theorem \ref{thm: prototypes H(1,1)}, and let $C$ be the simple
cylinder of the prototype. As before, we will say
that the prototype is 
{\em nondegenerate for $M$} if every non-critical vertical leaf passes
through $C$ infinitely many times, and every critical leaf passes
through the interior of 
$C$ at least once, and {\em degenerate} otherwise. 
Note that the vertical direction need no longer be
a completely periodic direction for $M$. 
Let $\sigma$ be the horizontal saddle connection which is
the top boundary segment of $C$. 
Let $R_b$ be as in \equ{eq: defn Rb}, define $s, t$ via \equ{eq: s t
  defined}, and 
define 
\eq{eq: defn R}{
R: [\lambda, b] \to [0,b), \ \ R(z) \df R_b(z-x+s).
}

\begin{prop}\name{prop: iet2}
The prototype with parameters $a,b,c,e,
\lambda, D,x,y$ is nondegenerate for $M$
 if and only if $x \neq 0, \, s \notin \{0, \lambda, b-\lambda\}$ and there are
 non-negative integers $k, \ell, m$ such 
that the following hold:
\eq{eq: first prong2}{
R^k(s) \in \left(0, \lambda \right)  \ \mathrm{and} \ 
R^j(s)\notin [0, \lambda] \ \mathrm{ for} \  j=1, \ldots, k-1,}
\eq{eq: second prong2}{
R^{\ell}(s+\lambda)\in \left(0, \lambda \right) \ \mathrm{and
  }\ R^j(s+\lambda )\notin [0, \lambda]\ \mathrm{for} \  j=1, \ldots,
  \ell-1}
and
\eq{eq: third prong}{
R^{m}(b)\in \left(0, \lambda \right) \ \mathrm{and
  }\ R^j(b)\notin [0, \lambda]\ \mathrm{for} \  j=1, \ldots,
  m-1.
} 
In this case the return map to $\sigma$
going upward along vertical leaves is an interval exchange on four
intervals, and the corresponding permutation is the `inversive permutation' $i \mapsto 5-i.$ 

\end{prop}  

\begin{proof}
The map $R$ in \equ{eq: defn R} is the return map to
the bottom of the non-simple cylinder of the prototype, for those $z$
which do not begin along the boundary with the simple cylinder $C$. 
As before, nondegeneracy is equivalent to the requirement that 
all vertical 
  separatrices going down from the singularity pass through the
  interior of 
  $C$. Let $k, \ell, m$ denote respectively the number of times
  the downward vertical separatrices $v_1, v_2, v_3$ respectively in 
  Figure \ref{fig: canonical polygon H(1,1)},
  pass through the non-simple cylinder in the 
  prototype, before reaching the interior of $C$. This leads to the formulae \equ{eq: first prong2},
\equ{eq: second prong2} and \equ{eq: third prong}. 

The three points in the interior of $\sigma$ are the three
discontinuities of the interval exchange which is the return map to
$\sigma$. Thus the 
  interval exchange has four intervals of continuity. Counting
  Weierstrass points one sees that there must be a Weierstrass point
  in the center of $C$. In particular $C$ is fixed by the involution. 
Symmetry of the vertical
foliation with respect to the hyperelliptic involution implies that
the permutation is inversive, as claimed. 
\end{proof}

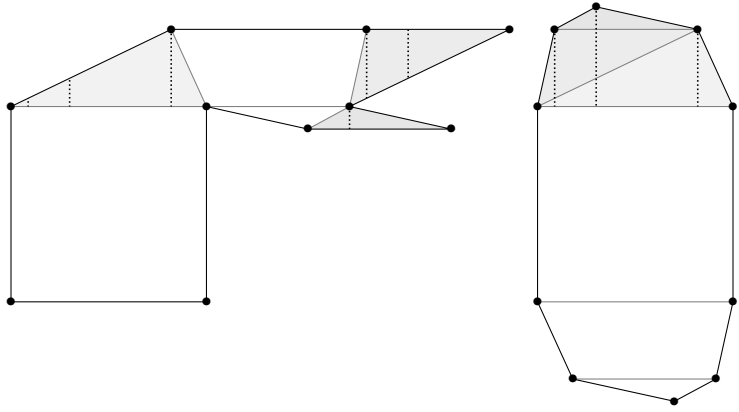
\begin{figure}[!h]
\begin{tikzpicture}
\begin{scope}[xshift=0cm,yshift=0cm,scale=1.5] 
\coordinate (A) at (0, 0);
\coordinate (H) at (1.4196, 0.6845);
\coordinate (G) at (3.151666, 0.6845);
\coordinate (F) at (4.4196, 0.6845);
\coordinate (E) at (3, 0);
\coordinate (B) at (0, -1.732);
\coordinate (C) at (1.732, -1.732);
\coordinate (D) at (1.732, 0);
\coordinate (I) at (2.63205, -0.2);
\coordinate (J) at (3.9, -0.2);
\begin{scope}
\path[clip] (A) -- (B) -- (C) -- (D) -- (I) -- (J) -- (E) -- (F) -- (H) -- 
(A) -- cycle;
\fill [gray!10] (A) -- (D) -- (H) -- cycle;
\fill [gray!15] (E) -- (F) -- (G) -- cycle;
\fill [gray!20] (I) -- (J) -- (E) -- cycle;
\draw [gray] (H) -- (D);
\draw [gray] (D) -- (E);
\draw [gray] (E) -- (G);
\draw [gray] (A) -- (D);
\draw [gray] (I) -- (E);
\draw [semithick,densely dotted] (G) -- +(0,-0.6845) ++(-3,0) -- +(0,-0.6845);
\draw [semithick,densely dotted] (H) -- +(0,-0.6845);
\draw [semithick,densely dotted]
(E) -- +(0,-0.6845) 
++(0.5196,0.6845) -- +(0,-0.6845)
(.5196,0.6845) -- +(0,-0.6845);
\end{scope}
\draw
(A) node {\scriptsize\textbullet} -- (H) node {\scriptsize\textbullet} -- 
(G) node {\scriptsize\textbullet} -- (F) node {\scriptsize\textbullet} --
(E) node {\scriptsize\textbullet} -- (J) node {\scriptsize\textbullet} --
(I) node {\scriptsize\textbullet} -- (D) node {\scriptsize\textbullet} --
(C) node {\scriptsize\textbullet} -- (B) node {\scriptsize\textbullet} --
(A) -- cycle;
\end{scope}
\begin{scope}[xshift=7cm,yshift=0cm,scale=1.5] 
\coordinate (A) at (0, 0);
\coordinate (H) at (1.4196, 0.6845);
\coordinate (G) at (0.151666, 0.6845);
\coordinate (F) at (4.4196, 0.6845);
\coordinate (E) at (0.5196, 0.8845);
\coordinate (B) at (0, -1.732);
\coordinate (C) at (1.732, -1.732);
\coordinate (D) at (1.732, 0);
\coordinate (I) at (2.63205, -0.2);
\coordinate (J) at (3.9, -0.2);
\coordinate (GG) at (1.580334, -2.4165);
\coordinate (EE) at (1.212334, -2.6141);
\coordinate (HH) at (0.3124, -2.4165);
\begin{scope}
\path[clip]
(A) -- (G) -- (E) -- (H) -- (D) -- (C) -- (GG)
-- (EE) -- (HH) -- (B) -- (A) -- cycle;
\fill [gray!10] (A) -- (D) -- (H) -- cycle;
\fill [gray!15] (A) -- (H) -- (G) -- cycle;
\fill [gray!20] (G) -- (H) -- (E) -- cycle;
\draw [gray] (G) -- (H);
\draw [gray] (A) -- (H);
\draw [gray] (A) -- (D);
\draw [gray] (B) -- (C);
\draw [gray] (GG) -- (HH);
\draw [semithick,densely dotted] (G) -- +(0,-0.6845);
\draw [semithick,densely dotted] (H) -- +(0,-0.6845);
\draw [semithick,densely dotted] (E) -- +(0,-0.8845);
\end{scope}
\draw
(A) node {\scriptsize\textbullet}
-- (G) node {\scriptsize\textbullet}
-- (E) node {\scriptsize\textbullet}
-- (H) node {\scriptsize\textbullet}
-- (D) node {\scriptsize\textbullet}
-- (C) node {\scriptsize\textbullet}
-- (GG) node {\scriptsize\textbullet}
-- (EE) node {\scriptsize\textbullet}
-- (HH) node {\scriptsize\textbullet}
-- (B) node {\scriptsize\textbullet}
-- (A) -- cycle;
\end{scope}
\end{tikzpicture}
\caption{From prototype to canonical decagon in $\HH(1,1)$.}\name{fig: canonical polygon H(1,1)}
\end{figure}

The canonical polygon for the prototype is constructed as before. It is a decagon. 
Let $(x_1, y_1), (x_2, y_2), (x_3, y_3)$ denote respectively the holonomies of the
segment going from the top left-hand corner of $C$ to the singularity
of $M$, by following the top of $C$ until the first (respectively
second and third) discontinuity of the interval exchange transformation
described in Proposition \ref{prop: iet}, and then up to the
singularity along the right hand edge of the first (resp. second and third) strip. 

The following propositions are proved analgously to Propositions
\ref{prop: canonical octagon} and \ref{prop: nec and suff}. We leave the proofs to the reader. 

\begin{prop}\name{prop: canonical decagon}
Let $k$ and $\ell$ be as in \equ{eq: first prong} and \equ{eq: second
  prong}. Then 
\eq{eq: formulae decagon}{
\left( \begin{matrix} x_1 \\ y_1\end{matrix}\right)
=\left( \begin{matrix} R^{\ell}(s+\lambda)  \\
 \ell(y+t)+t \end{matrix} \right), \ \ \left(\begin{matrix}x_2 \\ y_2
   \end{matrix}\right) = \left(\begin{matrix}  R^m(b) \\
m(y+t) \end{matrix}\right), \ \ 
\left( \begin{matrix} x_3 \\ y_3\end{matrix}\right)
=\left( \begin{matrix} R^k(s) \\
k(y+t)+t \end{matrix} \right),
}
and $x_1 < x_2 < x_3$. 
The canonical polygon is strictly convex if and only if 
\eq{eq: convexity decagon}{
\begin{split}
& x_2y_1-x_1y_2  > 0, \\ 
& x_2y_3-x_3y_2+x_3y_1-x_1y_3+x_1y_2-x_2y_1 > 0, \\
& x_2y_3-x_3y_2+ \lambda(y_2-y_3) >0. 
\end{split}
}
\combarak{compare with computer program.} 
\end{prop} 
\ignore{
\begin{proof}
\combarak{rewrite}
The points $(x_i, y_i), \ i=1,2,3$ are vertices of the decagon in light
of Proposition \ref{prop: iet2}. To find out the relative order of the
points $x_i$, we argue as before, identifying the downward vertical prongs in Figure
\ref{} with the downward vertical prongs in Figure \ref{} by moving
counterclockwise around the singularity at $(0,0)$ \combarak{give this
  a letter which corresponds to something in the figures.} 
As before, the three inequalities in \equ{eq: convexity decagon} correspond to internal angles less than
$\pi$ at each of the three vertices $(x_i, y_i)$,  convexity at the other
seven vertices following automatically from these. 
\end{proof}
As before we have:
}
\qed

\begin{prop}\name{prop: nec and suff2}
If $M$ is a lattice surface in $\HH(1,1)$ and has a strictly convex
presentation, then there is a prototype which is nondegenerate for
$M$, for which the given convex polygon presentation coincides with
the canonical polygon. 
\end{prop}
\qed

\section{Proof of Theorem \ref{thm: main}}
The main result of this section is:
\begin{prop}\name{prop: for main}
The list \equ{eq: list} is the complete list of discriminants for
which there is no prototype for which the canonical polygon is
strictly convex. 
\end{prop}

\begin{proof}[Deduction of Theorem \ref{thm: main} from Proposition
    \ref{prop: for main}]
The stratum $\HH(2)$ contains a surface with a strictly convex presentation,
viz. the regular octagon. Therefore there is an open subset of
$\mathcal{U}$ of $\HH(2)$ consisting of surfaces with a strictly
convex presentation. By Theorem \ref{thm: McMullen dynamics H(2)}, if $M$ is not a lattice
surface then its $G$-orbit intersects $\mathcal{U}$ and therefore $M$
has a strictly convex presentation. Therefore the first assertion of
Theorem \ref{thm: main} follows from Propositions \ref{prop: nec and
  suff} and \ref{prop: for main}. For the second assertion, note that
\equ{eq: list2} is obtained from \equ{eq: list} by removing all $D$
which are square. Thus it suffices to show that if $D$ is a square,
$M$ has a convex presentation,  and if $D$ is not a square and the
surface of discriminant $D$ has a convex presentation, it must be
strictly convex. The first part of this claim follows from 
Proposition \ref{Hubert Lelievre McMullen}. Now suppose $M$ has a
convex presentation which is not strictly convex, then the
corresponding polygon $\mathcal{P}$ must be either a parallelogram or a hexagon with
some vertices in the interiors of edges. If $\mathcal{P}$ is a parallelogram
then the surface has a one-cylinder presentation, and $D$ is a square
by Proposition \ref{Hubert Lelievre McMullen}. If $\mathcal{P}$ is a
hexagon, let $\theta$ be the direction of a side of $\mathcal{P}$. The
flow in direction $\theta$ is completely periodic by Theorem \ref{thm:
Calta McMullen H(2)}, and the return map to a diagonal of the hexagon
is an interval exchange on two intervals. This means that this return
map is a rational rotation, i.e. $\mathcal{P}$ maps to a one-cylinder
presentation with waist curve in direction $\theta$. Again using
Proposition \ref{Hubert Lelievre McMullen} we find that $D$ is a
square. 
\end{proof}
\subsection{Computations for even $D$} 
Given even $D$ let $\kappa$ be such that 
\eq{eq: range D}{
(2\kappa)^2 < D \leq (2\kappa+2)^2.
}
Set $e \df 2\kappa-6$ and 
\eq{eq: defn b}{
b \df \frac{D-e^2}{4}
}
so that 
\eq{eq: range b}{
6\kappa-9 < b \leq 8\kappa-8.}
Now define 
\eq{eq: defn lambda}{
\lambda \df \frac{e+\sqrt{D}}{2},
}
so that 
\eq{eq: range lambda}{
2\kappa-3 < \lambda \leq 2\kappa-2.
}
Let $a$ be the unique integer in the
  interval 
$\left[\frac{b+\lambda}{3}-1,
  \frac{b+\lambda}{3}\right).$ 
\begin{prop}\name{prop: computations, even}
The following hold when $\kappa \geq 7$:
\begin{itemize}
\item[(i)]
$\frac{b+\lambda}{3} - \frac{b-\lambda}{2} >1. $
\item[(ii)]
$3 \lambda<b$.

\item[(iii)] $b-\lambda< 2a < b$. 
\item[(iv)]
$3a>b$.
\item[(v)]
$a > \lambda.$
\item[(vi)]
$\lambda +a <b.$
\item[(vii)]
$3a-b <
\lambda.$ 
\end{itemize}
\end{prop}
\begin{proof}
It follows from \equ{eq: range b} and \equ{eq: range lambda} that 
$$
\frac{b+\lambda}{3} - \frac{b-\lambda}{2} > \frac{2\kappa -7}{6}$$
 which
implies (i), since $\kappa \geq 7.$ 
To prove (ii), define 
$x \df \sqrt{D}.$ 
We need to show
  $\frac{3\left(\sqrt{D}+e\right)}{2} < \frac{D- e^2}{4}.$ Plugging in $e = 2\kappa-6$
and $x = \sqrt{D}$, this is equivalent to 
\eq{eq: equivalent}{
x^2 -6x -4\kappa^2+12\kappa>0 \ \ \mathrm{for} \ x \in \left( 2\kappa, 2\kappa+2 \right],}
which is easily proved using the quadratic formula. 
The right hand side of (iii) follows from (ii) and the definition of $a$. The left hand side follows
from the definition of $a$ and (i). 
Since $\lambda \geq 3$, (iv) follows from the definition of $a$.
The inequality (v) follows from (ii) and (iv), and inequality (vi)
follows from (iii) and (v). Inequality (vii) is immediate from the
definition of $a$. 
\end{proof}

\subsection{Computations for odd $D$}
Given odd $D$ let $\kappa$ be such that 
\eq{eq: range D odd}{
(2\kappa-1)^2 < D \leq (2\kappa+1)^2.
}
Set $e_1 \df 2\kappa-7, e_2 \df 2\kappa-9$ and define $b_1, b_2$ via \equ{eq:
  defn b}, using $e_1, e_2$ respectively, so that 
\eq{eq: range b odd}{
6\kappa-12 < b_1 \leq 8\kappa-12, \ \ 8\kappa-20 < b_2 \leq 10\kappa-20.}
Define $\lambda_1, \lambda_2$ via \equ{eq: defn lambda}
so that 
\eq{eq: range lambda odd}{
2\kappa-5 < \lambda_2 \leq 2\kappa-4 < \lambda_1 \leq 2\kappa-3.
}
Let 
$a_1$ be the unique integer in the
  interval 
$\left[\frac{b+\lambda}{3}-1,
  \frac{b+\lambda}{3}\right)$ and let $a_2$ be the unique integer in
  the interval $\left[\frac{b_2+\lambda_2}{4}-1, \frac{b_2+\lambda_2}{4} \right)$.

\begin{prop}\name{prop: computations, odd}
Inequalities (i)---(vii) from Proposition \ref{prop: computations, even}
hold for $a_1, b_1, \lambda_1$ when $\kappa \geq 7$. If $\kappa \geq5$ then:
\begin{itemize}
\item[(i)]
$\frac{b_2+\lambda_2}{4} - \frac{b_2-\lambda_2}{3} >1. $
\item[(ii)]
$4 \lambda_2<b_2$. 
\item[(iii)] $b_2-\lambda_2< 3a_2 < b_2$. 
\item[(iv)]
$4a_2>b_2$.
\item[(v)]
$a_2 > \lambda_2.$
\item[(vi)]
$\lambda_2 +2a_2 <b_2.$
\item[(vii)]
$4a_2-b_2 <
\lambda_2.$ 
\end{itemize}
\end{prop}

\begin{proof}
Very similar to the proof of Proposition \ref{prop: computations,
  even}, and we leave the details to the reader. 
\combarak{There are a few details in the tex file as comments. Please
  check that you agree.}
\end{proof}

\subsection{There are strictly convex presentations  for all sufficiently large $D$}

\begin{cor}\name{cor: D big}
For any 
$D \geq 200,$ the lattice surface of discriminant $D$ has
a strictly convex presentation. 
\end{cor}
\begin{proof}
Suppose first that $D$ is even. Let $a, b, e$ be as in Proposition
\ref{prop: computations, even}, 
and set $c \df 1.$ Then one easily checks that \equ{eq: prototypes}
holds so 
that $D, a, b, c, e$ is a prototype. Inequalities
(v), (iii), (iv) and (vii) of Proposition \ref{prop: computations, even} show that $k=3$ satisfies \equ{eq: first
  prong}, and inequalities (iii) and (vi) imply
that
$\ell =2$ satisfies \equ{eq: second prong}. 
Finally,
\equ{eq: convexity octagon} follows from (ii). 

If $D$ is odd we can use either of the prototypes $D, a_1, b_1, 1,
e_1$ or $D, a_2, b_2, 1, e_2$, where in the first case we use $k=3,
\ell=2$ and in the second we use $k=4, \ell = 3.$ It is easy to check
using Proposition \ref{prop: computations, odd} that the hypotheses of
Proposition \ref{prop: canonical octagon} are satisfied. 
If $D \equiv 5 \mod 8$ then this concludes the proof. If $D \equiv 1
\mod 8$ there are two possible spins, so it remains to check that the
two cases $D, a_1, b_1, 1, e_1$ and $D, a_2, b_2, 1, e_2$ correspond
to different spins, and this is immediate from $e_1=e_2+2, c=1$ and \equ{eq: detecting
  spin}. 
\end{proof}

\subsection{A computer algorithm}\name{subsec: computer algorithm}
In order to deal with the remaining cases $5 \leq D < 200$, we have
implemented the following algorithm:

\medskip

{\tt
\begin{enumerate}
\item
Given $D$, enumerate all solutions $a,b,c,e$ to \equ{eq:
  prototypes}. 
\item
 For each prototype $a,b,c,e$ from step 1, do: 
\begin{enumerate}
\item
Let $k$ be the smallest $j$ so that $R_b(ja) \in [0,
\lambda]$, and let $\ell$ be the smallest $j$ so that $R_b(\lambda+aj)
\in [0, \lambda]$ (cf. 
\equ{eq: first prong} and \equ{eq: second prong}).
\item If $\{R_b(ka), R_b(\lambda+\ell a) \}
\cap \{0, \lambda\} \neq \varnothing$ then the prototype is degenerate for $M$. 

\item
Check whether
\equ{eq: convexity octagon} holds. If yes, stop: the
canonical octagon for this prototype is a
strictly convex presentation. 
\end{enumerate}
\item
If \equ{eq: convexity octagon}  fails for all prototypes, then
there is no strictly convex presentation. 
\end{enumerate}
}

\medskip

We have implemented this algorithm on Sage. The code and output of our
computation are available along with this paper at \cite{our
  paper}. A sample of the output of our program is below in Figure
\ref{fig: sample output}. Using the computation we find:

\begin{prop}\name{prop: computer}
For $5 \leq D < 200$, there is a prototype with discriminant $D$ whose
canonical octagon is convex, if and only if $D$ does not appear on the
list \equ{eq: list}. 
\end{prop}
\qed
\begin{proof}[Proof of Proposition \ref{prop: for main}]
Immediate from Corollary \ref{cor: D big} and Proposition \ref{prop: computer}. 
\end{proof}

\begin{figure}[!h]
\includegraphics[width=\linewidth]{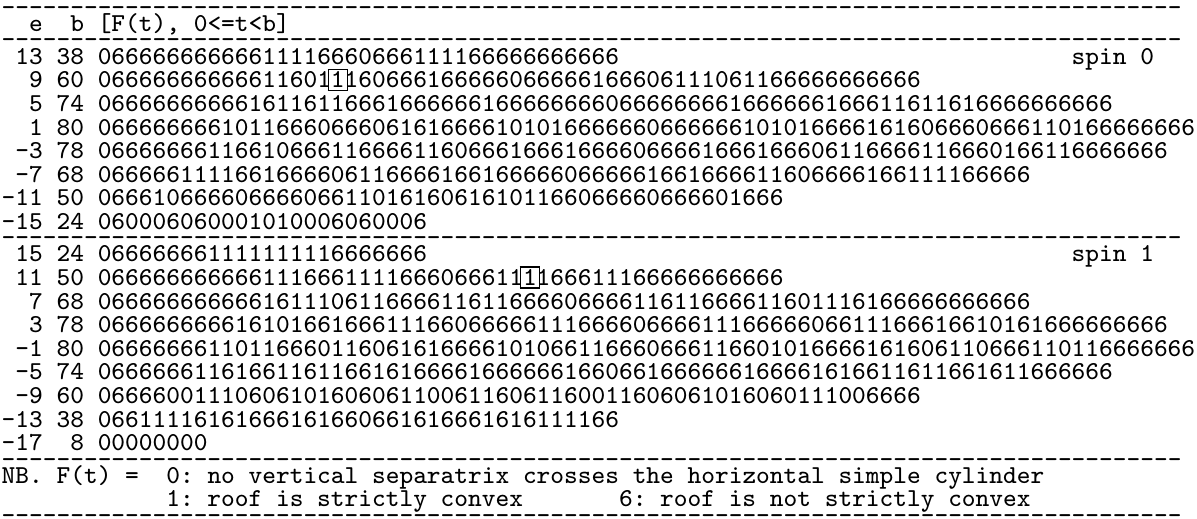}
\caption{Flowdown algorithm output for $D=321$, both spins. The boxed $1$'s 
are the convex presentations of Corollary \ref{cor: D big}}\name{fig: sample output}
\end{figure}

\section{The discriminants $D=4, 9, 16$ in $\HH(1,1)$}
\name{sec: 9 and 16}
\begin{prop}\name{prop: 4 9 16}
No surface in $\EE_{4,9,16}$ has a strictly convex presentation. 
\end{prop}
\begin{proof}
  Surfaces in $\EE_D$ with $D=d^2$ admit
  a degree $d$ cover onto a torus with two branch points,
  corresponding to the distinct singularities of the surface. Suppose
  $M \in \EE_{4,9,16}$ has a convex presentation and, applying the
  $G$-action and dilation, suppose it covers the standard square once-punctured torus
  with holonomy in $\Z^2$, via a degree $d$ covering map. Let $P$ be
  the convex decagon representing $M$, label its vertices cyclically
  counterclockwise, and let $P_0$ be the
  pentagon which is the convex hull of vertices labeled
  1,3,5,7,9. Then the vertices of $P_0$ all project to the same point on the torus. This
  implies that $P_0$ is a strictly convex
  pentagon in the plane, with vertices in $\Z^2$ and area strictly
  less than $d$. A further restriction on $P_0$ is that the sum of the internal
  angles at adjacent vertices of $P_0$ (i.e. vertices whose labeling
  on $P$ differs by 2) is strictly greater than $\pi$. This
  follows from an argument as in the last paragraph of the proof of
  Proposition \ref{prop: central symmetry}. We call this the {\em
    adjacent angles restriction.}

Thus to conclude the proof it suffices to show that for $d=3,4,5$, there is no $P_0$ with
the stated properties. Note that all properties are invariant under
the action of $2 \times 2$ integer matrices of deteminant $\pm 1$. 
Suppose first that there are 2 adjacent edges of $P_0$ which do not
contain lattice points on their interior. By applying an integer
matrix we can assume that three of the vertices of $P_0$ are $v_1 = (0,1), v_2 = (0,0), v_3
= (1,0)$. We write the two additional vertices $v_4, v_5$ as $(m,n)$
and $(k,\ell)$ respectively. Then applying the adjacent angles
restrictions on the various edges obtain the following inequalities:
\[
m > 1, \ \ n> 0, \ \ k>0, \ \ \ell>1, \ \ \ell > n, \ \ m>k, \ \ nk >
(m-1)(\ell-1). 
\]
From these we derive 
$$
k < m< 1+ \frac{nk}{\ell-1} \leq 1+k,
$$
which is impossible. 


For the remaining case we
will use Pick's formula again (see the proof of Proposition \ref{prop: interesting square tiled}). 
If no two adjacent edges of $P_0$ are without interior lattice points,
then the number of lattice points is at least 8 (5 at the vertices and
at least 3 on the edges). Additionally, we can apply an integer matrix
to assuming that three of the adjacent vertices of $P_0$ are $(k,0),
(0,0), (0, \ell)$ with $k\geq 2$ and $\ell\geq 1$. Then the adjacent
angle restrictions imply that the points $(1,1)$ and $(1,2)$ are
interior points of $P_0$. Thus by Pick's formula the area of $P_0$ is
at least 5. 
\end{proof}

\begin{remark}
The case $D=4$ could also be proved by using Corollary \ref{cor: short saddle connection} and
noting that all surfaces in $\EE_4$ have nontrivial translation
automorphisms. 
\end{remark}

It is possible to extend Proposition \ref{prop: 4 9 16} to higher genus, as
follows. For $g \geq 2$, and $d \geq
  g$, let $\mathcal{C}^{(d)}(g-1, g-1)$ denote the 
Hurwitz space of surfaces in the hyperelliptic component of the
stratum $\HH(g-1, g-1)$ which admit a degree $d$ 
cover of a torus branched over two points. Proposition \ref{prop: 4 9
  16} corresponds to the case $g=2, \, d_0=4$ of the following 
statement. 

\begin{prop}\name{remark: more general} 
For any $g\geq 2$ there is $d_0 =d_0(g) \geq g$ such that for $d  = g,
\ldots, d_0$, $\mathcal{C}^{(d)}(g-1, g-1)$
contains no surface with a 
strictly convex presentation. 
\end{prop}
\begin{proof} We first note that the restriction $d \geq g$ is imposed
 in order to ensure that  $\mathcal{C}^{(d)}(g-1, g-1) \neq
 \varnothing.$ Indeed, the Hurwitz formula implies that
 $\mathcal{C}^{(d)}(g-1, g-1) = \varnothing$ when $d<g$, and 
it is not hard to see using `staircases' that the loci
$\mathcal{C}^{(d)}(g-1, g-1)$ are nonempty for $d \geq g$. 
We will prove the statement for 
$$d_0(g) \df \max\{g, a(g)\}, \ \ \mathrm{where \ } a(g) \sim \frac{g^3}{54}.$$

A convex presentation, if it exists in $\mathcal{C}^{(d)}(g-1, g-1)$, is a $(2g+2)$-gon. Arguing as in
the proof of  Proposition \ref{prop: 4 9 16}, we find that it contains
a strictly convex $(g+1)$-gon $P_0$ with 
vertices at lattice points, and with area strictly less than
$d$. Moreover $P_0$ satisfies the following restriction (generalizing
the adjacent edges restriction): Let $\vec{v}_1, \ldots,
\vec{v}_{g+1}$ denote the holonomies of the edges of $P_0$, ordered
cyclically counterclockwise and equipped with the boundary orientation
of $P_0$. As usual we consider indices mod $g+1$. Then for $j=1, \ldots,
2g+1$ and for any $\displaystyle{1 \leq i < \frac{g+1}{2}}$, the angle 
between $\vec{v}_j$ and $\vec{v}_{j+i}$ (measured counterclockwise) is less than $\pi$, and if
$g+1$ is even, the angle between $\vec{v}_j$ and $\vec{v}_{j+(g+1)/2}$
is $\pi$. We need to show that no such $P_0$ exists. 

Let $a(g)$ denote the minimal area of a strictly convex $(g+1)$-gon with vertices
in $\Z^2$. Then it is known (see \cite{BT, Simpson}) that $a(g) \sim
\frac{g^3}{54}$ and $a(g) \geq g$ for $g \geq 9$. So there is
nothing to prove for $g \geq 9$. If $g = 3, \ldots, 8$ we need to show
any strictly convex $(g+1)$-gon with vertices
in $\Z^2$, of area less than $g$, does not satisfy the generalized
adjacent edges restriction. This can be done by a case-by-case
analysis involving Pick's formula, as in the proof of Proposition
\ref{prop: interesting square tiled}; we leave the details to our
dedicated reader. 
\end{proof}

\section{Finding strictly convex presentations in
  $\HH(1,1)$}\name{sec: finding}
\subsection{Pulling back from $\HH(2)$}
One can consider $\HH(2)$ as part of a bordification of $\HH(1,1)$,
where the boundary added corresponds to surfaces obtained by
`collapsing the two singularities.'  Since having a strictly convex
presentation is an open condition, one may expect that having a
strictly convex presentation in $\HH(2)$ implies having one nearby, in
$\HH(1,1)$. Indeed this turns out to be the case, as we will show in
this section. For a discussion of this point of
view, and a careful discussion of the inverse operation of `splitting
apart a zero', see \cite{EMZ}. 

Let $P$ be a strictly convex centrally symmetric octagon with vertices $v_0,
\ldots ,
v_7$, listed counterclockwise, and let  $M$ denote the surface in $\HH(2)$ obtained by
gluing opposite sides of $P$.  Denote by $e_j = v_{j+1}-v_j$ the
holonomies of saddle connections representing edges of $P$ 
(where indices are considered mod $8$).  Let $u$ be a vector
satisfying 
$$
e_0 \wedge u \ <\ 0 \ < \  e_7 \wedge u,
$$
where $v_1\wedge v_2$ denotes the determinant of the matrix with column
vectors $v_1,v_2$; in other words, $u$ lies in the convex cone generated
by $e_7$ and $e_0$. 
For any $\vre>0$, let $P_{u, \vre}$ denote the centrally symmetric decagon whose vertices
(listed counterclockwise) are the points
$$v_0, w_0, v_1, w_2, v_3, w_4, v_4, w_5, v_6, w_7 \ \ \mathrm{
  where} \ w_j \df v_j + \vre u.$$
That is we obtained $P_{u,\vre}$ by splitting each of the vertices
$v_0, v_4$ in two 
and replacing them by two nearby vertices separated by motion in the
direction of $u$; and all other vertices have been moved at most a
distance $\vre$ from their initial positions. 
Since $P$ is strictly convex, the choice of direction of $u$ ensures
that there exists
$\delta>0$ such that for any $\vre\in(0,\delta)$, $P_{u,\vre}$ is also
strictly convex. See Figure \ref{fig:octagon-decagon}.

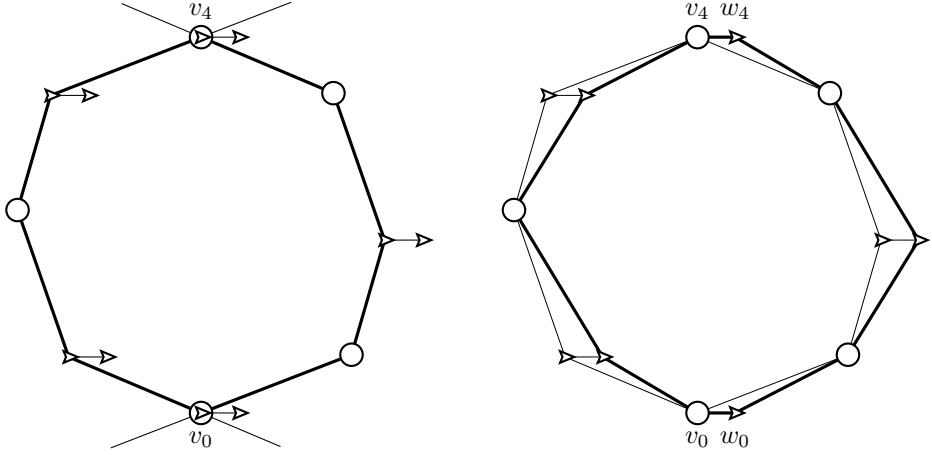
\begin{figure}[h]
\begin{tikzpicture}
\path
  (0,-2.50) node [inner sep=0pt] (v0) {}
  +(0.5,0) node [inner sep=0pt] (w0) {}
  (2,-1.725) node [inner sep=0pt] (v1) {}
  (2.44,-0.2) node [inner sep=0pt] (v2) {}
  +(0.5,0) node [inner sep=0pt] (w2) {}
  (1.76,1.755) node [inner sep=0pt] (v3) {}
  (0,2.5) node [inner sep=0pt] (v4) {}
  +(0.5,0) node [inner sep=0pt] (w4) {}
  (-2,1.725) node [inner sep=0pt] (v5) {}
  +(0.5,0) node [inner sep=0pt] (w5) {}
  (-2.44,0.2) node [inner sep=0pt] (v6) {}
  (-1.76,-1.755) node [inner sep=0pt] (v7) {}
  +(0.5,0) node [inner sep=0pt] (w7) {}
  ;
\draw [very thick]
  (v0) node [below=3pt] {$v_0$} -- (v1) -- (v2) -- (v3) -- (v4)
  node [above=3pt] {$v_4$} -- (v5) -- (v6) -- (v7) -- (v0) -- cycle;
\draw (v0) -- ++(-2*0.6,-0.775*0.6);
\draw (v0) -- ++(1.76*0.6,-0.745*0.6);
\draw (v4) -- ++(2*0.6,0.775*0.6);
\draw (v4) -- ++(-1.76*0.6,0.745*0.6);
\draw
  (v0) -- (w0)
  (v2) -- (w2)
  (v4) -- (w4)
  (v5) -- (w5)
  (v7) -- (w7);
\draw [thick]
  (v0) node [circle,draw,thick,fill=white,inner sep=3pt] {}
  (v0) node [dart,draw,thick,fill=white,inner sep=1pt] {}
  (w0) node [dart,draw,thick,fill=white,inner sep=1pt] {}
  (v1) node [circle,draw,thick,fill=white,inner sep=3pt] {}
  (v2) node [dart,draw,thick,fill=white,inner sep=1pt] {}
  (w2) node [dart,draw,thick,fill=white,inner sep=1pt] {}
  (v3) node [circle,draw,thick,fill=white,inner sep=3pt] {}
  (v4) node [circle,draw,thick,fill=white,inner sep=3pt] {}
  (v4) node [dart,draw,thick,fill=white,inner sep=1pt] {}
  (w4) node [dart,draw,thick,fill=white,inner sep=1pt] {}
  (v5) node [dart,draw,thick,fill=white,inner sep=1pt] {}
  (w5) node [dart,draw,thick,fill=white,inner sep=1pt] {}
  (v6) node [circle,draw,thick,fill=white,inner sep=3pt] {}
  (v7) node [dart,draw,thick,fill=white,inner sep=1pt] {}
  (w7) node [dart,draw,thick,fill=white,inner sep=1pt] {}
  ;
\begin{scope}[xshift=6.6cm]
\path
  (0,-2.50) node [inner sep=0pt] (v0) {}
  +(0.5,0) node [inner sep=0pt] (w0) {}
  (2,-1.725) node [inner sep=0pt] (v1) {}
  (2.44,-0.2) node [inner sep=0pt] (v2) {}
  +(0.5,0) node [inner sep=0pt] (w2) {}
  (1.76,1.755) node [inner sep=0pt] (v3) {}
  (0,2.5) node [inner sep=0pt] (v4) {}
  +(0.5,0) node [inner sep=0pt] (w4) {}
  (-2,1.725) node [inner sep=0pt] (v5) {}
  +(0.5,0) node [inner sep=0pt] (w5) {}
  (-2.44,0.2) node [inner sep=0pt] (v6) {}
  (-1.76,-1.755) node [inner sep=0pt] (v7) {}
  +(0.5,0) node [inner sep=0pt] (w7) {}
  ;
\draw [very thick]
  (v0) node [below=3pt] {$v_0$} -- (w0) node [below=3pt] {$w_0$}
  -- (v1) -- (w2) -- (v3) -- (w4) node [above=3pt] {$w_4$} -- (v4)
  node [above=3pt] {$v_4$} -- (w5) -- (v6) -- (w7) -- (v0) -- cycle;
\draw [very thin]
  (v0) -- (v1) -- (v2) -- (v3) -- (v4)
  -- (v5) -- (v6) -- (v7) -- (v0) -- cycle;
\draw [very thin]
  (v0) -- (w0)
  (v2) -- (w2)
  (v4) -- (w4)
  (v5) -- (w5)
  (v7) -- (w7)
  ;
\draw [thick]
  (v0) node [circle,draw,thick,fill=white,inner sep=3pt] {}
  (w0) node [dart,draw,thick,fill=white,inner sep=1pt] {}
  (v1) node [circle,draw,thick,fill=white,inner sep=3pt] {}
  (v2) node [dart,draw,thick,fill=white,inner sep=1pt] {}
  (w2) node [dart,draw,thick,fill=white,inner sep=1pt] {}
  (v3) node [circle,draw,thick,fill=white,inner sep=3pt] {}
  (v4) node [circle,draw,thick,fill=white,inner sep=3pt] {}
  (w4) node [dart,draw,thick,fill=white,inner sep=1pt] {}
  (v5) node [dart,draw,thick,fill=white,inner sep=1pt] {}
  (w5) node [dart,draw,thick,fill=white,inner sep=1pt] {}
  (v6) node [circle,draw,thick,fill=white,inner sep=3pt] {}
  (v7) node [dart,draw,thick,fill=white,inner sep=1pt] {}
  (w7) node [dart,draw,thick,fill=white,inner sep=1pt] {}
  ;
\end{scope}
\end{tikzpicture}
\caption{Octagon to decagon: $v_0$ and $v_4$ split, circle-shaped
vertices $v_j$ stay, dart-shaped vertices $v_j$ move to $w_j$.
Left: original octagon (thick) and cone of directions for $u$; right:
original octagon (thin) and resulting decagon (thick). Strict 
convexity is preserved.}
\label{fig:octagon-decagon}
\end{figure}

Now let $M_{u,\vre}$ be the surface obtained by identifying opposite
sides of $P_{u,\vre}$. 
According to  \cite[Cor. 5.6]{McMullen-SL(2)}, the question of whether a genus two surface is an
eigenform of discriminant $D$ depends only on the absolute period map
$H_1(M, \Z) \to \R^2$ obtained by integrating the planar 1-forms $dx,
dy$ on $M$. Note that for any $\vre>0$, the absolute period map of
$M_{u, \vre}$ is the same as that of $M$. Therefore 
$M_{u,\vre}\in\HH(1,1)$ is
an eigenform of discriminant $D$ whenever $M$ is.

This implies:
\begin{prop}\name{prop: oct to dec}
For any $D$, if there is a lattice surface in $\HH(2)$ with
discriminant $D, D_0$ or $D_1$, with a strictly convex presentation, then there is
an eigenform surface of discriminant $D$ in $\HH(1,1)$ with a strictly
convex presentation.  
\end{prop}
\qed

\subsection{A remaining finite list of cases}
\begin{prop}\name{prop: remaining cases}
For the following values of $D$, the eigenform locus $\EE_D \subset
\HH(1,1)$ contains surfaces with a strictly convex presentation:
\eq{eq: list3}{
5, 12, 17, 21, 25, 32, 36, 41, 45,
49, 64, 77, 81.
}

\end{prop}
Note that there are two differences between \equ{eq: list} and
\equ{eq: list3}: discriminants 9 and 16 appear in \equ{eq: list} and
not in \equ{eq: list3}; and all cases of the form $D_0, D_1$ are
written as $D$.

\begin{proof}
For the case $D=5$ we have the regular decagon. 
Next we deal with the cases $17, 41,49,81 $. In these cases, one of
the symbols $D_0, D_1$ appears in the list \equ{eq: list} but not the
other. For example, \equ{eq: list} contains $41_0$ but not
$41_1$. Suppose with no loss of generality that $D_0$ is not on the
list \equ{eq: list}, i.e. there is a lattice surface in $\HH(2)$ corresponding to
discriminant $D_0$, with a strictly convex presentation. Applying
Proposition \ref{prop: oct to dec} we conclude that $\EE_D$
contains a surface with a strictly convex presentation. 

For the nine remaining cases we have found strictly convex presentations by
a computer search. Namely we have written a computer program, similar
to the one described in \S \ref{subsec: computer algorithm}, which, given a prototype
$a,b,c,e,x,y$ in $\HH(1,1)$, checks whether the canonical decagon for
this prototype is convex. This algorithm is based on the discussion of
\S \ref{subsec: canonical H(1,1)}. For each fixed $D$, there are finitely many possibilities
for the discrete values $a,b,c,e$ of the
prototype. For each fixed value of these, the computer program
searched with a large finite number of evenly spaced 
values of $x$ and $y$ and stopped when a strictly convex canonical
decagon was found. Parameters for strictly convex polygons obtained by
this computer search are shown in
Figure \ref{table H(1,1)}. The computer code and
results of computations can be accessed\footnote{Do {\em not} see the
  references therein! Due to risk of infinite regress, the reader is advised to
    proceed with care.} at \cite{our paper}. 
\end{proof}
\begin{figure}[!h]
\begin{tabular}{@{}rrrrrrr@{}}
\toprule
\multicolumn{1}{c}{$D$} &
\multicolumn{1}{c}{$a$} &
\multicolumn{1}{c}{$b$} &
\multicolumn{1}{c}{$c$} &
\multicolumn{1}{c}{$e$} &
\multicolumn{1}{c}{$x$} &
\multicolumn{1}{c}{$y$}\\
\midrule
12 & 0 & 3 & 1 & 0 & 0.9 & 0.2\\
21 & 0 & 5 & 1 & 1 & 2.0 & 0.2\\
25 & 0 & 3 & 2 & -1 & 0.7 & 0.2\\
32 & 0 & 7 & 1 & 2 & 2.8 & 0.1\\
\bottomrule
\end{tabular}\qquad\qquad
\begin{tabular}{@{}rrrrrrr@{}}
\toprule
\multicolumn{1}{c}{$D$} &
\multicolumn{1}{c}{$a$} &
\multicolumn{1}{c}{$b$} &
\multicolumn{1}{c}{$c$} &
\multicolumn{1}{c}{$e$} &
\multicolumn{1}{c}{$x$} &
\multicolumn{1}{c}{$y$}\\
\midrule
36 &  5 &  9 &  1 &  0 &  1.0 &  0.1\\
45 & 0 & 9 & 1 & 3 & 3.5 & 0.1\\
64 &  6 &  16 &  1 &  0 &  0.5 &  0.1\\
77 & 0 & 13 & 1 & 5 & 5.8 & 0.1\\
\bottomrule
\end{tabular}
\caption{Parameters for strictly convex presentations in some $\EE_D$'s.}\name{table H(1,1)}
\end{figure}

\section{Unstable convexity in $\EE_{4,9,16}$ }\name{sec:
  unstable convexity}
We have observed that the set of surfaces admitting a strictly convex
presentation is open. 
In this section we will prove a partial converse for the eigenform
loci $\EE_4, \EE_9$ and $\EE_{16}$.  We will show
that if $M$ is represented by a decagon $P$ which is convex but not
strictly convex, then for almost every surface $M' \in \EE_{4,9,16}$ which is
sufficiently close to $M$, the corresponding polygon $P'$ is nonconvex. 
We begin by making this 
 precise. 

Recall that for any $D$, the eigenform locus $\EE_D$ is
a closed suborbifold of $\HH(1,1)$, of real dimension 5.  
A {\em marked decagon} is a decagon (not necessarily convex) equipped with a partition of its edges into
pairs, such that paired edges are parallel and of equal
length. There is an obvious topology on the space of marked
decagons. Identifying paired edges maps each marked decagon $P$ to a 
translation surface $M_P$, and this map $P
\mapsto M_P$ is a 
homeomorphism between a neighborhood of $M$ in $\HH(1,1)$ and a
neighborhood of $P$ in the space of marked decagons. We will denote
this neighborhood by $\mathcal{U} = \mathcal{U}(M,P)$, and denote the
inverse map by $M \mapsto P_M$. Thus, if $M=M_P \in \mathcal{U}$
and $M' \in \mathcal{U}$ is close to $M$, then  $P' = P_{M'}$ is close
to $P$ and 
$M'=M_{P'}$. 

\begin{prop}\name{prop: unstable decs}
Suppose $D=4,9$ or $16$. Suppose $P$ is a convex marked decagon such
that $M_P \in \EE_D$ and $\mathcal{U} = \mathcal{U}(M,P)$ is as
above. Then there is a neighborhood $\mathcal{U}_0 \subset
\mathcal{U}$ of $M$ such that 
$$
\{M' \in \mathcal{U}_0 \cap \EE_D : P_{M'}
\mathrm{\ is \ convex } \}
$$
is a four-dimensional submanifold.  
\end{prop}
We preface the proof of this proposition with some remarks about the
structure of strata and the structure of the loci $\EE_D$. More
details can be found in \cite{zorich survey} and \cite{McMullen-SL(2)}. Let
$S$ be a model compact orientable surface of genus two with a set
$\Sigma = \{\sigma_1, \sigma_2\}$ consisting of
two distinguished points of $S$. The stratum
$\HH(1,1)$ is locally modelled on the cohomology group $H^1(S, \Sigma;
\R^2)$, and the eigenform locus $\EE_D$ is modelled on the semi-direct
product group $G \ltimes \R^2$
(where $\R^2$ represents motion in the rel leaves). More precisely, for
any $M \in \HH(1,1)$ there is a neighborhood $\mathcal{U}$ of $M$ in
$\HH(1,1)$, and a homeomorphism of $M$ and $S$ mapping the
singularities of $M$ to $\Sigma$,  such
that each $M' \in \mathcal{U}$ is uniquely determined by the
holonomies $\hol(\alpha, M')$ obtained by integrating the translation
structure along $\alpha$, for any curve $\alpha$ in $S$ representing a
relative one-cycle in $H_1(S, \Sigma)$. Moreover (making $\mathcal{U}$
smaller if necessary), there is a small neighborhood
$\mathcal{V}$ 
of the identity in $G \ltimes \R^2$, such that for any $M' \in \EE_D \cap \mathcal{U}$ 
there are unique $(g,v) \in \mathcal{V}$, such that for any 1-cycle
$\alpha$, 
$$\hol(\alpha, M') = \left\{\begin{split} g. \, \hol(\alpha, M) &  \ \ \
    \mathrm{if} \ \alpha
    \ \mathrm{is \ a \ closed \ curve \ in \ } S \\ g. \, (\hol(\alpha, M) +v)
    & \ \ \     \mathrm{if} \ \alpha \ \mathrm{goes \ from \ } \sigma_1 \mathrm{\ to \ } \sigma_2 \end{split} \right. $$
Here $g.\vec{u}$ denotes the image of $\vec{u} \in \R^2$ under the
linear map $g$. We will denote the surface corresponding to $(g,v)$ as
above by $(g,v).\, M$.

The idea of proof will be to show that given any $P$, there is a
one-dimensional linear subspace $L \subset \R^2$, such that if $(g,v)
\in \mathcal{V}$ is as above, and $v \notin L$, then the decagon
$P_{M'}$ corresponding to $M' = (g,v). \, M$ is not convex. That is, all
motions in the rel leaf except possibly motions in a particular direction
destroy convexity. We proceed to the details. 

\begin{proof}
We take $\mathcal{U}_0$ small enough so that it is
contained in the neighborhood as in the preceding paragraph. That is, the notation $(g,v). \, M$ makes sense in
$\mathcal{U}_0$, for $(g,v)\in \mathcal{V}$, and so does the notation
$P_{M'}$. It follows from the preceding discussion that $P_{(g,v).M} $
is obtained from $P_M$ as follows. We first subdivide the vertices of $P$
into two classes, according to whether they map to the singularity
labelled $\sigma_1$ or
$\sigma_2$ in $M$. We then keep the vertices mapping to $\sigma_1$ fixed and move
all of the others by adding $v$. Connecting the vertices, this
defines a new polygon $P(v)$ if $v$ is sufficiently small (and we can assume
this by making $\mathcal{U}_0$ smaller). We then apply the map $g$ to
$P(v)$. The resulting polygon is $P_{(g,v).M}$. 

Since linear maps in $G$ preserve convexity, in order to prove the
proposition it suffices to show that for any convex marked decagon
$P$, for which $M_P \in \EE_{4,9, 16}$, the set of $v$ for
which $P(v)$ is convex, are contained in a one-dimensional linear
subspace. To show this, it suffices to show that for any $P$, there are
three or more edges of $P$ which are consecutive and parallel, and
connect distinct singularities; for then any rel perturbation $v$ which is
not in the direction of these edges will result in $P(v)$ being non-convex. E.g. in the last polygon in 
Figure \ref{marked decagons} (a square), this occurs in the top and
bottom of the figure. If we move the black points up (respectively
down) with respect to
the white points, convexity will be destroyed at the vertex between
edges $b$ and $c$ (respectively between $a$ and $b$) in the top of the
figure, so any rel operation which does not move points
horizontally will destroy convexity. 

\begin{figure}[h]
\includegraphics[scale=0.5]{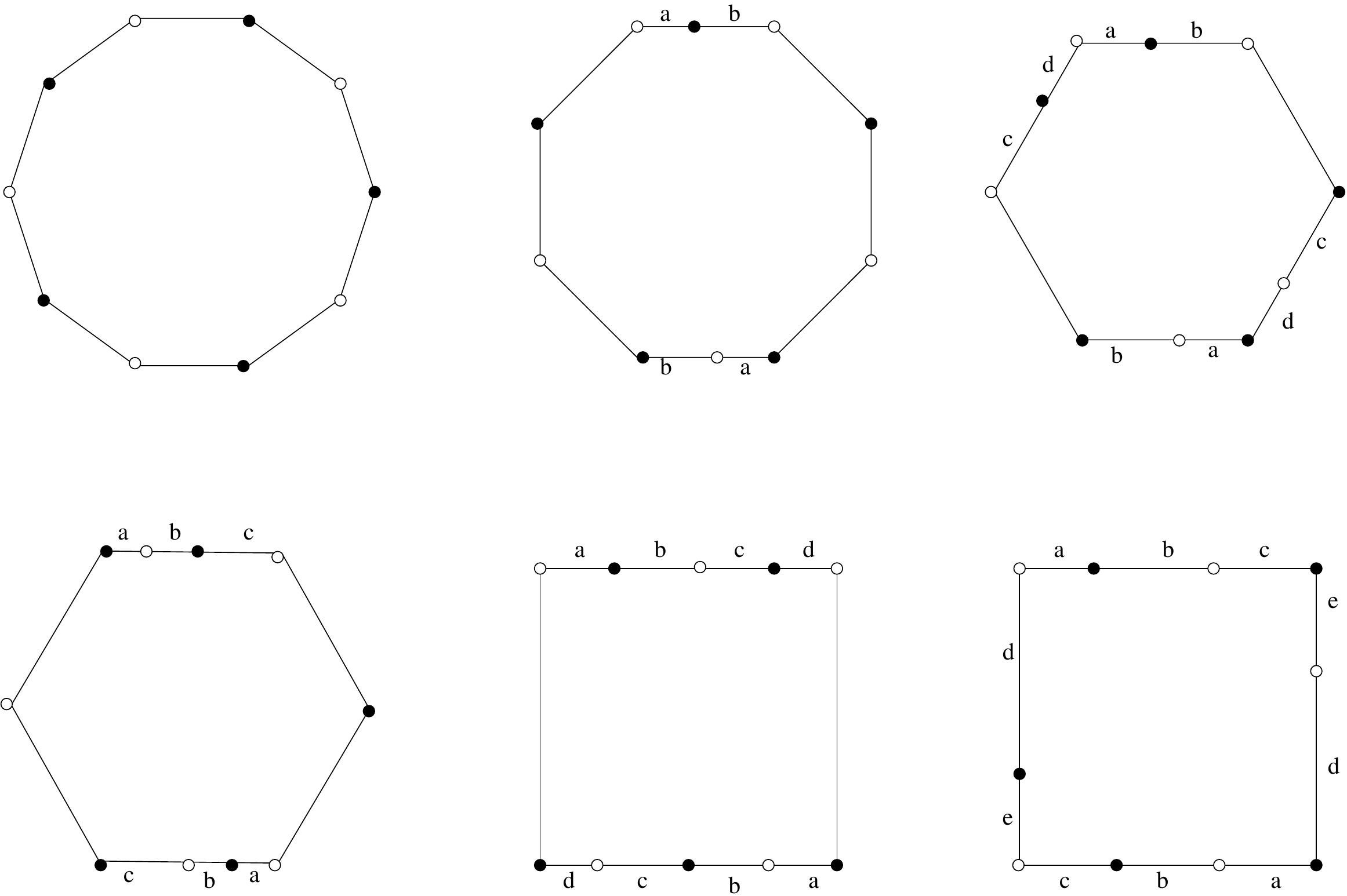}
\caption{The standard marked decagon, with several convex realizations
  in $\HH(1,1)$. The top row cannot appear in $\EE_{4,9, 16}$.}
\label{marked decagons} 
\end{figure}

The proof of existence of three parallel consecutive edges proceeds by
a case-by-case analysis. We will show 
that the only marked decagons which represent surfaces in $\HH(1,1)$
are those shown in Figures \ref{marked decagons} and \ref{marked
  decagons1}. Since, by \S\ref{sec: 9 and 16}, surfaces in $\EE_{4,9, 16}$ have no strictly
convex presentations, the decagon which is the first polygon in Figure
\ref{marked decagons} does not appear in these eigenform
loci. Similarly, the octagon and hexagon in the top row of Figure
\ref{marked decagons} also do not appear in $\EE_{4,9, 16}$,
since they can be made into a strictly convex decagon by a rel
perturbation, which would again contradict \S \ref{sec: 9
  and 16}. All
the other polygons in the figure satisfy our claim. This will conclude
the proof.

\begin{figure}[h]
\includegraphics[scale=0.5]{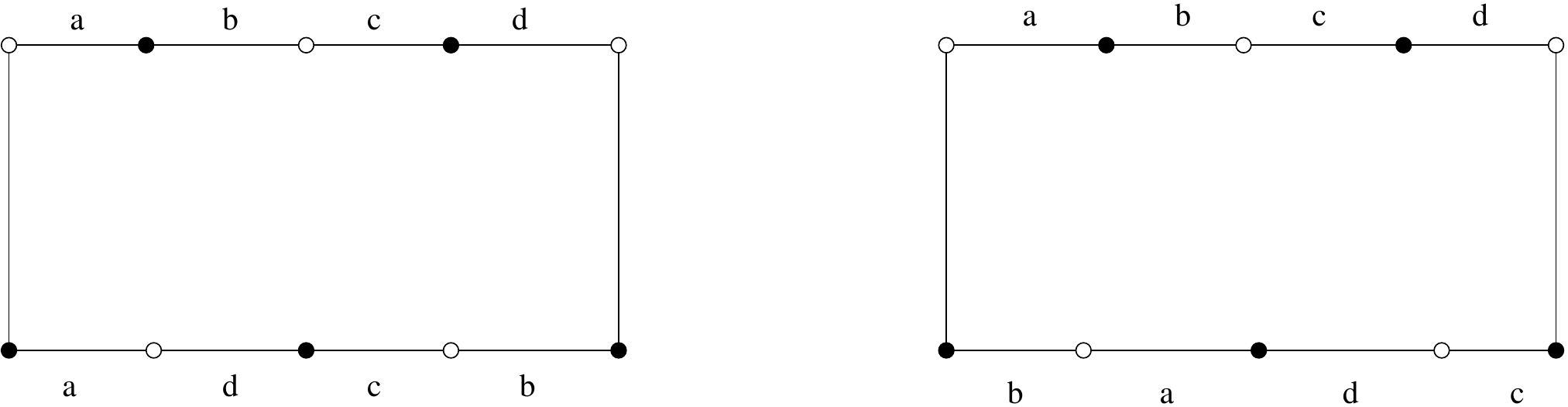}
\caption{Two additional marked  decagons, realized as cylinders
  in $\HH(1,1)$.}
\label{marked decagons1} 
\end{figure}

When realized as a convex polygon in the plane, some sides of the
marked decagon become parallel. This gives a partition of the set of
edges into subsets of parallel edges, where each subset is composed of
some pairs of edges identified in the marked decagon (but possibly
more than one pair). Considering the different possibilities
for these partitions, one finds the following possibilities: 
\begin{itemize}
\item[(a)]
a strictly convex decagon. 
\item[(b)]
a strictly convex octagon in which a pair of parallel edges are each a concatenation of
two consecutive edges of $P$.  
\item[(c)]
a strictly convex hexagon in which two pairs of parallel edges are each a concatenation
of two consecutive edges of $P$. 
\item[(d)]
a strictly convex hexagon in which a pair of parallel edges are each a
concatenation of three consecutive edges of $P$.   
\item[(e)]
a parallelogram in which a pair of parallel edges is each a
concatenation of four consecutive edges of $P$. 
\item[(f)]
a parallelogram in which one pair of parallel edges is each a concatenation
of three consecutive edges of $P$, and one pair of parallel edges is each a
concatenation of two consecutive edges of $P$. 
\end{itemize}
We will go over these cases one by one and show that all markings
which give rise to a surface in $\HH(1,1)$ appear in Figures
\ref{marked decagons} and \ref{marked decagons1}. Note that the total cone angle around
singularities for a surface in $\HH(1,1)$ is $8\pi$ and hence all
vertices of the decagon must be singularities of the translation
surface. All indices in the discussion below will be given in
cyclic order along the boundary of $P$. Also note that if $a_1, a_2$
are consecutive and $b_1, b_2$ are consecutive and parallel then we cannot glue
$a_1$ to $b_2$ and $a_2$ to $b_1$ because then the vertex between
$a_1$ and $a_2$ will have cone angle $2\pi$,
i.e. will not be a singularity, and we will not be in $\HH(1,1)$. We
will call this the {\em basic restriction}. 

Clearly case (a) corresponds to the first polygon in Figure \ref{marked
  decagons}. In case (b), suppose $a_1, a_2$ and $b_1, b_2$ are two
pairs of consecutive edges whose concatenations are parallel edges of
the octagon. If $a_1$ is glued to $b_1$ and $a_2$ to $b_2$ in $P$ then
we have the second polygon in Figure \ref{marked decagons}. The other case is forbidden by the basic
restriction. 
By the same reasoning, in case (c), we can only glue sides as in the
third polygon in Figure \ref{marked decagons}.

In case (d), let the three pairs of consecutive edges be labelled
$a_1, a_2, a_3$ and $b_1, b_2, b_3$. The gluing in Figure \ref{marked decagons} corresponds to
identifying $a_i$ with $b_i$. If $a_1$ is identified with $b_3$ then
by the basic restriction we must identify $a_2$ with $b_2$ and $a_3$
with $b_1$. Now consider the vertex on $a_3$ which is not adjacent to
$a_2$. Following the identifications we see that it is identified with
three other vertices and has a total angle of $\pi$, so we are not in
$\HH(1,1)$. By similar reasoning $a_3$ cannot be identified with
$b_1$. Since $a_1$ and $a_3$ cannot both be identified with $b_2$, we
have that for $i=1$ or $i=3$, $a_i$ is identified with $b_i$. But then
the basic restriction forces the gluing to be as in Figure \ref{marked
  decagons}.

In case (e), all three cases of Figures  \ref{marked decagons} (fifth
polygon) and \ref{marked
  decagons1} are possible. To see that there are no others, let $a_1,
\ldots, a_4, b_1, \ldots, b_4$ be the labelling of the four
consecutive edges. If $a_1$ is identified with $b_4$ and $a_4$ with
$b_1$ then the point vertex on  $a_1$ on the side opposite to $a_2$
becomes a regular point and we are not in $\HH(1,1)$. If $a_1$ is
identified with $b_4$ and $a_4$ with $b_2$ then  by the basic
restriction we must identify $a_3$ with $b_3$ 
and hence $a_2$ with $b_1$, which violates the basic restriction for
the pair $a_3, a_4$. If $a_1$ is identified with $b_4$ and $a_4$ with
$b_3$ then by the basic restriction for the sides $a_2, a_3$ we have
the polygon on the left side of Figure \ref{marked decagons1}. So we
have dealt with all cases in which $a_1$ is identified 
with $b_4$, and by symmetry, when $a_4$ is identified with $b_1$. Now
suppose $a_1$ is identified with $b_2$ and $a_4$ with $b_3$. By the
basic restriction $a_2$ is identified with $b_4$ and $a_3$ with
$b_1$. Then the vertex between $b_2$ and $b_3$ 
becomes a regular point and we are not in $\HH(1,1)$. The remaining
case is $a_1$ identified with $b_3$ and $a_4$ with $b_2$. If $a_2$ is
identified with $b_4$ we have the polygon on the right side of Figure
\ref{marked decagons1} and if $a_2$ is identified with
$b_1$ then the vertex between $a_2$ and $a_3$ becomes a regular point.  

In case (f), applying the basic restriction to the pair of consecutive
sides $a_1, a_2$, $b_1, b_2$ which form a side of the hexagon, we see
that we must identify $a_i$ with $b_i$. Then the discussion reduces to
case (d).
\end{proof}

\combarak{I imagine this proof could be made much shorter, please feel
  free to fiddle with it.}

\begin{cor}\name{cor: measure zero}
Let $D=4, 9$ or $16.$ 
With respect to the measure of Theorem \ref{thm: McMullen eigenform},
almost every surface in $\EE_D$ has no convex presentation. 
\end{cor}

\begin{proof}
In order to prove that a set of translation surfaces has measure zero,
it is enough to show that it has measure zero in the cover of
$\HH(1,1)$ corresponding to translation surfaces equipped with a
{\em marking}, i.e. a fixed identification of $M$ with $S$ such that
singularities map to $\Sigma$. 
Fixing a marking,  we note that the possible edges of convex
polygons $P$ may be labelled by homotopy classes of curves beginning
and ending at $\Sigma$. In particular there are countably many  ways of
using a marked decagon $P$ to get a convex presentation. For each such
marked decagon $P$ and each surface $M$ constructed via $P$ with a
convex presentation, we have exhibited a neighborhood of $M$ in which
the set of surfaces for which $P$ is convex, is of measure zero. The
result follows using the fact that a countable union of sets of
measure zero has measure zero. 
\end{proof}

\section{Proof of theorem \ref{thm: main H(1,1)}}
Assertions (i) and (ii) were proved in \S \ref{sec: 9 and 16} and \S \ref{sec: finding}
respectively. To deduce (iii) from (ii), note that if $D$ is not a
square, and is not equal to 5, by Theorem \ref{thm: McMullen dynamics
  H(1,1)}, any $G$-orbit in $\EE_D$ is dense in $\EE_D$. Moreover in
case $D=5$ the only orbit which is not dense in $\EE_5$ is that of the
regular decagon, which is clearly strictly convex. Thus the $G$-orbit
of any surface which is neither arithmetic, nor contained in
$\EE_{4,9, 16}$, contains a surface with a strictly convex
presentation. This concludes the proof in light of $G$-invariance. 
Assertion (iv) follows from Corollary \ref{cor: EMM finitely many}. 
\ignore{
For (iv) we use Theorem \ref{thm: equidistribution}. Suppose by
contradiction that there are infinitely many surfaces $M_1, 
M_2, \ldots $ without strictly convex presentations, lying on distinct
$G$-orbits, and such that the corresponding discriminants $D_i$ are
not equal to 4, 9 or 16. 
By Theorem \ref{thm: main H(1,1)}(ii), each $D_i$ is a square, and
$\EE_{D_i}$ contains a nonempty
open subset consisting of surfaces with a strictly convex
presentation. In particular there is an open subset in $\HH(1,1)$
consisting of surfaces with a strictly convex presentation. Therefore,
choosing an open set $\mathcal{U}$ according to 
whether or not there is an infinite sequence of $i$ such that $D_i$ is
constant,  we obtain a contradiction to \equ{eq: i large enough}. 
}
To prove (v) note that the collections of lattice surfaces in
$\EE_{4,9, 16}$ consists of a countable collection of $G$-orbits, so in
particular is of measure zero. In view of Corollary \ref{cor:
  measure zero}, almost
every surface in $\EE_{4,9, 16}$ has the required properties. 
\qed

\section{Another question of Veech} \name{sec: another question}
Theorem \ref{thm: main H(1,1)}(v) resolves a question of Veech
\cite[Question 5.2]{Veech hyperelliptic}. Namely Veech asked whether
there are non-lattice surfaces with no convex presentations, and as we
have seen, almost every surface in $\EE_{4,9,16}$ has these
properties. In case of an affirmative answer, Veech asked a more
refined question. Namely (in our notations), given $g$, let $\HH$
denote the hyperelliptic
component of one of the strata $\HH(g-1, g-1), \HH(2g-2)$, let
$\Hnc$ denote the subset of surfaces without convex
presentations, and let $\overline{\Hnc}$ denote the closure of $\Hnc$. Is there a surface in $\Hnc$  whose  
orbit-closure is $\overline{\Hnc}$? Our analysis shows that in the
stratum $\HH(1,1)$, the answer to this question 
is negative. Namely, by Theorem \ref{thm: main H(1,1)}(iii, iv) $\Hnc$ consists
of: 
\begin{enumerate}
\item[(a)]
Almost every surface in  $\EE_{4,9,16}$. 
\item[(b)]
A finite list of arithmetic surfaces.
\end{enumerate}
Therefore $\overline{\Hnc}$ consists of $\EE_{4, 9, 16}$, and a finite list of closed $G$-orbits. 
Since $\EE_4, \EE_9$ and $\EE_{16}$ are closed and disjoint, part (a) of this
list shows a negative answer to Veech's question in $\HH(1,1)$. We
remark that the list in (b) is nonempty. In light
of Proposition \ref{prop: interesting square tiled}, we know that
there are at least three $G$-orbits of
arithmetic surfaces outside $\EE_{4, 9, 16}$ without strictly convex
presentations. Namely, this will be satisfied by the $G$-orbit of any
square tiled surface made of $d \leq 7$ squares, which is not in $\EE_{4,9,16}$. Since 5 and 7 are
prime we can use $d=5,7$ to get two examples in $\EE_{49}$
and one in $\EE_{25}$.

\end{document}